\documentclass[a4paper]{amsart}

\usepackage[english]{babel} 
\usepackage{fullpage}
\usepackage{afterpage}
\usepackage{hyperref}
\usepackage[T1]{fontenc}
\usepackage[utf8]{inputenc}
\usepackage{lmodern}
\usepackage{amsmath,amsthm,amsfonts}
\usepackage{mathrsfs} 
\usepackage{graphicx} 
\usepackage{multicol}

\usepackage{tikz} 
\usepackage{tikz-cd} 
\usetikzlibrary{patterns}
\usepackage{caption}
\usepackage{pgfplots} 
\usepackage{wrapfig}
\usepackage{amssymb}
\usepackage{array}
\usepackage{stackrel}
\usepackage{extarrows}
\usepackage{pgfplots}
\usepackage{mathtools}
\usepackage{algpseudocode}
\usepackage{algorithm}

\usepackage{amsmath,amsthm,amsfonts}
\usepackage{mathrsfs} 
\usepackage{graphicx} 
\usepackage{multicol} 
\usepackage{tikz} 
\usepackage{tikz-cd} 
\usepackage{caption}
\usepackage{pgfplots} 
\usepackage{wrapfig}
\usepackage{amssymb}
\usepackage{array}
\usepackage{stackrel}
\usepackage{extarrows}
\usepackage{pgfplots}
\usepgfplotslibrary{polar}
\pgfplotsset{compat=newest}
\usetikzlibrary{calc}
\usetikzlibrary{graphs}
\usetikzlibrary{fit}
\usepgfplotslibrary{polar}
\pgfplotsset{compat=newest}
\usetikzlibrary{calc}

\title{Cohomology of pure symmetric automorphisms of right-angled Artin groups}
\author{Peio Ardaiz-Gale}



\thanks{\noindent
The author is partially supported by the Spanish Government PID2021-126254NB-I00, PID2024-155800NB-C32, and also  by Universidad P{\'u}blica de Navarra (grant code: Plan de promoci{\'o}n de grupos. ``{\'A}lgebra. Aplicaciones'').} 

\newcommand{\Z}{\mathbb{Z}}
\newcommand{\G}{\Gamma}

\newcommand{\B}{\mathcal{B}}
\newcommand{\au}{\mathcal{A}}
\newcommand{\C}{\mathcal{C}}
\newcommand{\s}{\mathcal{S}}
\newcommand{\A}{{A_\Gamma}}

\newcommand{\PAut}{\Sigma\mathrm{PAut}}
\newcommand{\POut}{\Sigma\mathrm{POut}}

\newcommand{\IA}{\mathrm{IA}}

\newcommand{\MM}{\mathrm{MM}}

\newcommand{\Wh}{\mathrm{Wh}}

\newcommand{\lk}{\mathrm{lk}}
\newcommand{\st}{\mathrm{st}}
\newcommand{\supp}{\mathrm{supp}}
\newcommand{\suppa}{\mathrm{supp}(\au)}
\newcommand{\peri}{\mathrm{Peripheral}}
\newcommand{\peria}{\mathrm{Peripheral}(\au)}
\newcommand{\ca}{\mathcal{C}(\au)}
\newcommand{\Stab}{\mathrm{Stab}}
\newcommand{\Inn}{\mathrm{Inn}}

\newcommand{\OO}{ \mathcal{O}}
\newcommand{\boundellipse}[3]
{(#1) ellipse (#2 and #3)}
\newtheorem{teo}{Theorem}[section]
\newtheorem{prop}[teo]{Proposition}
\newtheorem{cor}[teo]{Corollary}
\newtheorem{lem}[teo]{Lemma}
\newtheorem{conj}[teo]{Conjecture}
\newtheorem*{teoA}{Theorem A}
\newtheorem*{teoB}{Theorem B}
\newtheorem*{teoC}{Theorem C}

\newtheorem*{conjA}{Conjecture A}

\theoremstyle{definition}
\newtheorem{defi}[teo]{Definition}
\newtheorem{ej}[teo]{Example}

\theoremstyle{remark}
\newtheorem*{nota}{Remark} 

\begin{document}

\begin{abstract}
We compute the cohomology groups of the pure symmetric outer automorphism group $\POut(A_\Gamma)$ and the pure symmetric automorphism group $\PAut(A_\Gamma)$ of a right-angled Artin group $A_\Gamma$. Using the equivariant spectral sequence arising from the action of $\POut(A_\G)$ on the generalized McCullough-Miller complex $\MM_\G$, we show that $H^q(\POut(A_\G))$ is free abelian and we compute its rank in terms of the combinatorics of certain poset. Applying the Lyndon-Hochschild-Serre spectral sequence and the Leray-Hirsch theorem we do the same for 
$H^q(\PAut(A_\G))$. In both cases the cohomology algebra is generated in degree 1. Finally, we introduce the Generalized Brownstein-Lee Conjecture, proposing a presentation of $H^*(\PAut(A_\G))$, and prove that it holds in dimension $2$.
\end{abstract}

\maketitle

\section{Introduction}
Right-angled Artin groups (RAAGs) and their automorphism groups are very active field of study in geometric group theory. Given a RAAG $\A$, the automorphisms of  $\A$ that map each standard generator to a conjugate of itself are called {\sl pure symmetric} and  form a subgroup denoted $\PAut(A_\G)$ that contains the subgroup of inner automorphism, so we can define the {\sl pure symmetric outer} automorphism group as $\POut(A_\G)\cong\PAut(\A)/\Inn(\A)$.   In the case when $A_\G$ is the free group $F_n$, $\PAut(F_n)$ is the fundamental group of the configuration space of $n$ unknotted, unlinked circles in the sphere $S^3$ \cite{{JMcCM}}, there is also a version of this for symmetric automorphisms of other right-angled Artin groups \cite{BB22}.  

In \cite{MM}, McCullough and Miller constructed a contractible simplicial  complex $\MM_n$  that admits an action of the group \( \POut(F_n) \) with free abelian stabilizers. The construction is based on the Whitehead poset $\Wh_n$, whose elements  can be expressed as  labeled trees. In \cite{JMcCM}, Jensen,  McCammond and  Meier use the McCullough-Miller complex to compute $H^*(\POut(F_n))$ and $H^*(\PAut(F_n))$, the  cohomology groups of  the (outer) pure symmetric automorphisms of the free group; even more, they prove the Brownstein-Lee conjecture \cite{BrLee}, which stablish a presentation of the cohomology ring $H^*(\PAut(F_n))$. 

In \cite{ArMar}, Ardaiz-Gale and Martínez-Pérez generalized both the Whitehead poset and the construction of McCullough and Miller from the case of \( \POut(F_n) \) to \( \POut(\A) \). This new complex, $\MM_\G$, admits an  action of the group \( \POut(\A) \) with free abelian stabilizers and it  is also  contractible. The powerful combinatorial machinery derived from the labeled trees in $\Wh_n$ is the key to prove the main results in \cite{JMcCM}. However, vertices in $\Wh_\G$ cannot be expressed as labeled trees in general, only as families of based partitions, so most arguments need to be adapted.

In this paper we will generalize some combinatorial results from $\Wh_n$ to $\Wh_\G$ and use the equivariant spectral sequence derived from the $\POut(\A)$ action over $\MM_\G$  to compute the cohomology groups $H^q(\POut(A_\G))$ and $H^q(\PAut(A_\G))$. We show (see Sections \ref{sec:MM} and \ref{sec:basis} for definitions):

\begin{teoA}
  Let $K_q$ be the number of rank $q$ essential vertex types in $\Wh_\G$, then
    $$
       H^q(\POut(A_\G))=\mathbb{Z}^{K_q}. 
    $$
    Besides, the cohomology ring $H^*(\POut(A_\G))$ is generated by the one-dimensional classes $\bar{\gamma}_I^j$, where  the $C_{I}^j$ are canonical automorphisms.   \label{teoA}
\end{teoA}

\begin{teoB} \label{teoB}
Let $K_i$ be the number of rank $i$ essential vertex types in $\Wh_\G$ and $N_j$ the number of $j$-cliques in $\G$, then
     $$
H^q(\PAut(A_\G))\cong \bigoplus_{i+j=q} \mathbb{Z}^{K_i\cdot N_j}.
    $$
     Besides, the cohomology ring $H^*(\PAut(A_\G))$ is generated in degree 1.
\end{teoB}

Furthermore, we generalize the relations from the Brownstein-Lee Conjecture to the case of RAAGs, formulate the Generalized Brownstein-Lee Conjecture, and prove that it holds in dimension 2.

\begin{conjA}[Generalized Brownstein-Lee Conjecture]\label{ConjA}
Let $\gamma_A^i$ be a one-dimensional class in
$H^*(\PAut(A_\G))$, where $A$ is a connected component in $\G-\st(v_i)$. For each pair $v_i,v_j$ of two non-adjacent vertices in $\G$, let $D^j, D^i$ be the respective dominant components and $\mathcal{C}$ a shared component. Then $H^*(\PAut(A_\G))$ is generated by the one-dimensional classes $\gamma_A^i$, modulo the following relations for each shared component $\C$ of $v_i,v_j$.
\begin{itemize}
        \item[(1)] $\gamma_A^i\wedge\gamma_A^i=0$
        \item[(2)] $\gamma_A^i\wedge\gamma_B^j=-\gamma_B^j\wedge\gamma_A^i$
        \item[(3)] $\gamma_{D^j}^i\wedge \gamma_{D^i}^j=0$ 
        \item[(4)] $\gamma_{\mathcal{C}}^i\wedge \gamma_{\mathcal{C}}^j=\gamma_{\mathcal{C}}^i\wedge \gamma_{D^i}^j+ \gamma_{D^j}^i\wedge\gamma_{\mathcal{C}}^j$ 
    \end{itemize}
\end{conjA}

We will denote by $R$ the ring given by the presentation in the previous  conjecture.

\begin{teoC}
     There exists a ring epimorphism
    $$
    \phi:H^*(\PAut(A_\G)) \twoheadrightarrow R
    $$
which is an isomorphism in degrees 1 and 2.
    
    \label{teo:H2_1}
\end{teoC}

Another important evidence about this conjecture can be derived from the work of Martínez-Pérez and Mendonça in \cite{MarMen}. 
In Theorem A in that paper it is shown that when the graph $\G$ does not contain four pairwise non-adjacent vertices $v_1,v_2,v_3,v_4\in\G$ that lie in four distinct connected components of $\G-\cap_{i=1}^4\lk(v_i)$, then the group $A_\Gamma$ is of Koszul type, implying that $\mathcal{U}_\G$, the universal enveloping algebra of the rational Lie algebra associated to the descending central series of the group is quadratic and has as quadratic dual the cohomology ring of the group. It turns out \cite{MarMen} that for arbitrary $\Gamma$ the quadratic dual of $\mathcal{U}_\G$ is precisely the ring with the presentation of Conjecture A. So this implies that if $\Gamma$ satisfies the condition above, then the rational version of Conjecture A holds true.

The paper is organized as follows: In Section \ref{sec:MM} we explain some basic results about MM$_\G$, obtained from \cite{ArMar}. In Section \ref{sec:3}, we go over the main definitions in \cite{JMcCM}. In Section \ref{sec:basis} we generalize the notion of \textit{one-two automorphism} which plays a key role in \cite{JMcCM} and do some technical work that will be needed later. In Section \ref{sec:PSO} we apply the previous results to the spectral sequence derived from the action of $\POut(A_\G)$ on MM$_\G$ to compute $H^*(\POut(A_\G))$ and prove Theorem A. In Section \ref{sec:PSA}, we use the Lyndon-Hoschild-Serre  spectral sequence associated to $
1\to A_\G\to \PAut(A_\G)\to \POut(A_\G) \to 1
$ and the Leray-Hirsch Theorem to compute the $H^*(\PAut(A_\G))$ groups and prove Theorem B. In Section \ref{sec:ring} we take the first steps into proving the Brownstein-Lee conjecture for $\PAut(A_\G)$ and we prove Theorem C.

{\sl Acknowledgments:} The author would like to thank Ric Wade for helpful conversations during a stay in Oxford that made spectral sequences less ghostly, and Conchita Martínez-Pérez for her invaluable help.

\section{The McCullough-Miller complex}\label{sec:MM}

In this chapter we will explain some basic results about the McCullough-Miller complex for RAAGs. A more detailed explanation can be found at \cite{ArMar}.
\begin{defi}[Graph]
    We will use $\Gamma$ to denote a finite simplicial graph with vertex set $V(\G)$ and set of edges $E(\G)$. Very often we will identify $\Gamma$ or a subgraph of $\Gamma$ with its vertex set. Let $v\in V(\G)$,  $\lk(v)$ denotes the link of $v$, which is the subgraph induced by the vertices adjacent to $v$. We will define the star of $v$, $\st(v)$, as the subgraph induced by $\lk(v)\cup\{v\}$.
\end{defi}

Associated to $\Gamma$ we may define a group as follows.

\begin{defi}[RAAG]
     The RAAG (right-angled Artin group) defined by a simplicial graph $\G$ is the group with presentation
    $$
A_{\Gamma}=\langle v \in V(\Gamma) \mid[v, w]=1 \textrm{ whenever } \{v, w\} \in E(\Gamma)\rangle
$$
\end{defi}
This is, $\A$ is the group generated by the vertices of $\Gamma$, which are called {\sl standard generators} so that two standard generators commute if and only if they are adjacent in $\Gamma$. Two simple examples are the extreme cases: if $\G$ is a graph with no edges, then $A_\G$ is a free group, and if $\G$ a complete graph, then $A_\G$ is a free abelian group.

\subsection{Pure symmetric automorphisms of RAAGs}

\begin{defi}[Pure symmetric (outer)automorphisms] Let $A_\Gamma$ be a RAAG. The group of pure symmetric automorphisms $\PAut(A_\G)$ of $A_\G$ is the group consisting of those automorphisms of $A_\G$ that map each standard generator  (i.e., each vertex of $\G$) to a conjugate of itself. Obviously, it contains the inner automorphism and the corresponding external version is denoted $\POut(A_\G)=\PAut(A_\G)/\Inn(A_\G).$
\end{defi}

Next, we introduce an important family of elements of $\PAut(A_\G)$.

\begin{defi}[Partial conjugations]
Given $v \in L= V(\G)\cup V(\G)^{-1}$ and $A$ a union of connected components of $\Gamma-\st(v)$, the {\sl partial conjugation} of $A$ by $v$ is the automorphism $C_{A}^v$  given by
$$
\left\{\begin{array}{l}
C_{A}^v(w)=vwv^{-1}, \text{ if } w \in A \\
C_{A}^v(z)=z, \text{ if } z \notin A
\end{array}\right.
$$   
where $w,z\in\G$.
\end{defi}

\begin{teo}[Laurence \cite{Lau}] The set of partial conjugations is a generating system for $\PAut(A_\G)$.   \label{teo:partialGen}
\end{teo}

\begin{nota}
    For the rest of the paper, we will only consider RAAGs $\A$ such that $\G$ has no dominating vertex. Equivalently, RAAGs with trivial center. There is no loss of generality in this assumption, which will be helpful in Section \ref{sec:PSA}, because, by Laurence's Theorem,  $\PAut(\A)\cong\PAut(A_{\G'})$ if $\G$ is obtained from $\G'$ by removing all the dominating vertices. \label{note:centre}
\end{nota}

When we are working with  non adjacent vertices $u$ and $v$, it will be really useful to think about $\G-\st(u)$ in terms of graph components. We will use the next notation due to Day and Wade \cite{DayWade}. 

\begin{defi}[Graph components: shared, dominant and subordinate]
     Let $u, v\in\Gamma$ be not linked.  A connected component of $\Gamma-\st(u)$ that is also a connected component of $\Gamma-\st(v)$ is said to be {\sl shared}. The unique component of $\Gamma-\st(u)$ that contains $v$ is called the {\sl dominant} component of $\G-\st(u)$ respect to $v$, we will denote it by $D^v$ when $u$ is clear. Finally, a {\sl subordinate} component is any connected component of $\Gamma-\st(u)$ that is contained in the dominant component of $\G-\st(v)$ respect to $u$. By Lemma 2.1 in \cite{DayWade}, any connected component of $\Gamma-\st(u)$ is of one of these types. Observe that if $A$ is a connected component of $\Gamma-\st(u)$ which is either shared or subordinate, then $A\cap\st(v)=\emptyset$.
     \label{defi:comp}
\end{defi}

Using the terminology in the previous definition, one can rewrite the presentation of Koban and Piggott (\cite{KoPi}) as follows:
\begin{teo}\label{teo:conmu}
    The group $\PAut(A_{\Gamma})$ is the group generated by all partial conjugations $C_{A}^u$ for $u\in \G$ and $A$ a connected component of $\G-\st(u)$, subject to the following relations:
 \begin{itemize}
     \item[i)] $\left[C_{A}^u, C_{B}^v\right]=1$ if  $u \in \st(v)$,
     \item[ii)]$\left[C_{A}^u, C_{B}^v\right]=1$ if $u \notin \st(v)$ and either $A$ and $B$ are shared but distinct or at least one of $A,B$  is subordinate,
     \item[iii)] $\left[C_{A}^u C_{B}^u, C_{A}^v\right]=1$ if $u \notin \st(v)$, $A$ is shared and $B$ dominant. \label{teo:presentation}
 \end{itemize}
\end{teo}

\begin{defi}[Valid based partitions, petals and length]
Let $u\in\G$ be a vertex. A {\sl $\G$-valid based partition with operative factor $u$} is a partition of the form $\tau_u=\{\{u\},P_1,\ldots,P_k\}$ of the set $\G-\lk(u)$ so that each $P_i$ is a union of connected components. Each $P_i$ is called a petal of the based partition.   The number of petals, $k$, is the 
    {\sl length} of the based partition, and we denote it $l(\tau_u)$. \label{defi:partitions}
\end{defi}

From now on, every based partition we use will be a $\Gamma$-admissible based partition, for simplicity we will simply say based partition if the graph $\G$ is clear.

\begin{defi}[Crossings, disjoint partition and compatibility]
We say that two based partitions $\tau_u$ and $\tau_v$ {\sl cross} if $u\notin \st(v)$ and there are petals $P$ of $\tau_u$ and $Q$ of $\tau_v$ such that  $D^v\not\subset P$, $D^u\not\subset Q$, and $P\cap Q\neq\emptyset$. It is easy to check that this is equivalent to asking that there is at least one shared component $\C$ such that $\C\subseteq P$ and $\C\subseteq Q$. 
 We will say that two based partitions $\tau_u, \tau_v$  are {\sl disjoint} if $u\notin \st(v)$ and they do not cross 
 and we say that they are {\sl compatible} if they are disjoint or $u\in\st(v)$.
\end{defi}

\begin{defi}[Vertex type]
    We define a {\sl vertex type} $\tau=\{\tau_{v_1},...,\tau_{v_n}\}$ as a collection of pairwise compatible based partitions, such that there is one for each vertex in $\G$.

    We will say that two vertex types $\tau$ and $\tau'$ are {\sl compatible} if their based partitions are pairwise compatible. And we will say that a collection of vertex types is {\sl compatible} if they are pairwise compatible.
\end{defi}

\begin{nota}
   Here we are using a slightly different notation from \cite{ArMar}, where vertex types where denoted $\underline{\underline{A}}$, and a based partition $\underline{A}_u$. The extensive notation in \cite{ArMar} is forced by  the technical work done there.  Here, this simpler version will be enough. 
\end{nota}

\begin{defi}[Order]\label{def:order}
    Let $v\in\Gamma$ and $\tau_v$, $\tau'_v$ two based partitions with operative factor $v$.  We say that $\tau_v\leq \tau_v'$ if every petal of $\tau_v'$ is contained in a petal of $\tau_v$, i.e., when every petal of $\tau_v$ is a union of petals of $\tau_v'$. In particular, if $\tau_v\leq \tau_v'$, then $l(\tau_v)\leq l(\tau'_v)$. 
    For vertex types, we will say that $\tau\leq \tau'$ if $\tau_v\leq \tau'_v,\ \textrm{ for every } v\in V(\Gamma)$.
\end{defi}

We have the next Lemma from \cite[Lemma 4.5.]{ArMar}:

\begin{lem}\label{lem:ordercomp} Let $\tau$ and $\tau'$ be vertex types with $\tau<\tau'$, then $\tau$ and $\tau'$ are compatible.
\end{lem}

\begin{defi}[$\Gamma$-Whitehead poset]
    The {\sl $\Gamma$-Whitehead poset} Wh$_{\G}$ is the poset formed by all the vertex types associated to $\G$ under the partial order previously defined. 
\end{defi}

\begin{defi}[Nuclear vertex]
The {\sl nuclear vertex} type, denoted  $\OO$, is the vertex type such that  for each $v\in \G$ the associated based partition is trivial, i.e., $\OO_v=\{\{v\},P\}$, where $P=\G-\st(v)$. It is the unique minimal element in the $\Wh_{\G}$ poset.
\end{defi}

\begin{defi}[Rank of a vertex type]
   We denote by $l(\tau_v)$ the length of a based partition, i.e., the number of petals it has. Then, the {\sl rank of a vertex type $\tau$} is 
   $$\mathrm{rk}(\tau)=\displaystyle\sum_{v\in\G}(l(\tau_v)-1)$$
\end{defi}

\begin{defi}[Carried automorphism]\label{def:carried}
  The group of automorphisms carried  by a based partition $\tau_u$ is  the subgroup of $\PAut(\A)$ generated by the partial conjugations $C^u_I$ where $I$ runs over the petals of $\tau_u$. The group of automorphisms carried  by a vertex type $\tau$ is the subgroup of $\PAut(\A)$ generated by the automorphisms carried by each based partition $\tau_u\in\tau$.
\end{defi}

The following is an easy  observation that will be useful later.

\begin{lem}\label{lem:ordercarrier} Let $\tau\in\Wh_\G$ be a vertex type. The subgroup of $\PAut(\A)$   carried  by $\tau$ contains the inner automorphisms. Moreover, given two vertex types $\tau$ and $\tau'$ such that $\tau<\tau'$, the subgroup of automorphisms carried by $\tau$ is inside the subgroup of automorphisms carried by $\tau'$.
\end{lem}

In what follows, we will often consider the subgroups of outer automorphisms carried by either a given based partition or a vertex type, which are just the corresponding quotients over $\Inn(A_\G)$.

\begin{defi}[Marked vertex type and the action of $\PAut(A_\G)$] We say that a generating system $X$ of $A_\G$ is a {\sl basis} if there is some $\alpha\in\PAut(A_\G)$ such that $X=\alpha(V(\G))$.
A {\sl marked based partition} is a pair $(X,\tau_u)$ where $X$ is an ordered basis of $A_\G$ and $\tau_u$ is a based partition and a {\sl marked vertex type} $(X,\tau)$ is the family of marked based partitions $(X,\tau_u)$ where $\tau_u$ runs over the set of based partitions of a vertex type $\tau$. By marking a based partition $\tau_u$ we understand that we are relabeling the elements in the partition using $X$ (via the automorphism $\alpha$). We extend the partial order of Definition \ref{def:order} to marked vertex types as follows: $(X_1,\tau)\leq(X_2,\tau')$ if $X_1=X_2$ and $\tau\leq\tau'$. The marked vertex types form a poset and the group $\PAut(A_\G)$ acts on this poset via $$\alpha(X,\tau)=(\alpha(X),\tau)$$
where $\alpha\in\PAut(A_\G)$.
We say that $\alpha$ is {\sl carried by a marked vertex type} $(X,\tau)$ if $\gamma\alpha\gamma^{-1}$ is carried by $\tau$ where $\gamma(X)=V(\G)$.
\end{defi}

Next, we define an equivalence relationship in the set of marked vertex types as follows.

\begin{defi}
     We say that two marked vertex types $(X_1,\tau)$ and $(X_2,\tau)$ are equivalent and we put  $(X_1,\tau)\equiv (X_2,\tau)$ if
there exists an $\alpha\in\PAut(A_\G)$ carried by $(X_1,\tau)$ such that $\alpha(X_1)=X_2$. We will denote by $[X_1,\tau]$ the equivalence class of $(X_1,\tau)$ under this relation. 
\end{defi}

\begin{defi}[$\G$-McCullough-Miller poset]
    The {\sl $\G$-McCullough-Miller poset} is the poset formed by the set of equivalence classes under the equivalence relationship $\equiv$ defined above with the induced partial order.
    The group $\PAut(A_\G)$ acts on this poset via $$\alpha([X,\tau])=[\alpha(X),\tau]$$
    and this action factors through the epimorphism $\PAut(A_\G)\to\POut(A_\G)$ so we also have an action of $\POut(A_\G)$.
\end{defi}

The $\G$-McCullough-Miller poset is formed by copies of the $\Gamma$-Whitehead poset attached to each other through the equivalence relation $\equiv$. It is easy to verify that  nuclear vertex types only carry inner automorphisms, so there is a bijection between the set of nuclear vertices $[X,\OO]$ in the $\G$-McCullough-Miller poset and the set of conjugacy classes of bases of $A_\G$.

\begin{defi}[$\MM_\Gamma$ complex]
    We define the {\sl McCullough-Miller complex $\MM_\G$ for a RAAG $A_\Gamma$} as the simplicial realization of the $\G$-McCullough-Miller poset. In the case when $\A$ is the free group with rank $n$ we will denote this complex by $\MM_n$. \label{def:MM}
\end{defi}

\begin{nota}
The  definition of the action above implies that the stabilizer $H=\Stab_{\POut(\A)}[V(\G),\tau]$  of the vertex $[V(\G),\tau]$ under the $\POut(\A)$ action is precisely the subgroup generated by those automorphisms carried by the vertex type $\tau$. This implies that the stabilizer of the vertex $[X,\tau]$ is $\alpha H\alpha^{-1}$ where $\alpha$ is such that $X=\alpha(V(\G))$. It is proved in \cite[Theorem A]{ArMar}  that these stabilizers are free abelian, that their rank can be computed in terms of the length of the based partitions of $\tau$ and that they yield in fact all the cell stabilizers of the  McCullough-Miller complex $\MM_\G$. 

     Taking this into account, one sees that the McCullough-Miller complex $\MM_\G$ can be also defined as follows. Recall that for a poset $\mathcal{F}$ of subgroups of a group $G$, one can define a new poset, called the {\sl coset poset}, as
$$C(\mathcal{F})=\{gH\mid H\in\mathcal{F},g\in G\}.$$
Then, the McCullough-Miller poset is the coset poset of the family 
$$\{\Stab_{\POut(A_\G)}[V(\G),\tau]\mid \tau\text{ vertex type}\}$$
of subgroups of $\POut(A_\G)$ and $\MM_\G$ is the associated geometric realization. \label{nota:stab}
\end{nota}

\begin{defi}[Stab$(\tau)$]
    We denote by Stab$(\tau)$  the stabilizer of the vertex type $\tau$ of $\MM_\G$ under the $\POut(A_\G)$ action. By definition, we have that Stab$(\tau)$ is the external version of  the subgroup of $\PAut(\A)$ carried by $\tau$. If $\sigma$ is a simplex in $|\Wh_\G|$ corresponding to a chain $\tau^0<\cdots<\tau^k$ then the stabilizer of $\sigma$ is the stabilizer of $\tau^0$ so we have $\Stab(\sigma)=\Stab(\tau^0)$. 
\end{defi}

As explained in the previous remark, the stabilizer of a vertex type is a free abelian group. Even more, the rank of a vertex type $\tau$ as an element of the poset $\Wh_\G$ coincides with the rank of Stab$(\tau)$ as a free abelian group.

\section{The \textit{one-two basis}}
\label{sec:3}

Our main goal is to compute $H^*(\POut(A_\G))$, and we will do it using an  equivariant spectral sequence for the  $\POut(\A)$ action on $\MM_\G$.  If  a group acts simplicially on  a contractible complex, then the first page of this equivariant spectral sequence can be computed in terms of the cell stabilizers (Chapter VII, Proposition 7.3 in \cite{Brown}). A key step in this process is choosing a \textit{good} basis for these stabilizers. We will do it in a way that generalizes the choice of \cite{JMcCM}. 

As proved in  \cite{ArMar} (Lemma 5.3.) and explained in the Remark after Definition \ref{def:MM}, the stabilizer of a vertex $\tau$ is precisely the group of automorphisms carried by $\tau$. By Definition \ref{def:carried}, the partial conjugations $C_I^u$, where $I$ runs over the petals of the based partitions $\tau_u$ of $\tau$, are a generating set for the stabilizer. Observe  that for a  based partition $\tau_u=\{\{u\},I_1,...,I_k\}$, the composition $C_{I_1}^u\cdots C_{I_k}^u$ is the inner automorphism associated to $u$. Since we are working with  outer-automorphisms, this means that the set $\{C_{I_1}^u,\cdots, C_{I_k}^u\}$ is not independent. So, if we want a minimal generating set, we have to choose which one to leave aside. This choice problem is solved as follows in \cite{JMcCM} for the case of the pure symmetric outer-automorphisms of free group:

\begin{defi}[One-two automorphism]
    Let $\{x_1,...,x_n\}$ be a fixed basis for  the free group of rank $n$ $F_n$. A \textit{one-two automorphism} is a partial conjugation $C_I^{x_j}$ such that:
    \begin{enumerate}
        \item[1)] $x_1\notin I$. 
        \item[2)] if $x_j=x_1$, then $x_2\notin I$. 
    \end{enumerate} \label{defi:one-two_aut}
\end{defi}
In simple words, a one-two automorphism never conjugates $x_1$ and never conjugates $x_2$ by $x_1$. The last important definition that we need from \cite{JMcCM} is the notion of  \textit{essential} vertex type.

\begin{nota}
    In $\Wh_n$, a vertex type $\tau$ can be expressed as a labeled tree, and the order in the Whitehead poset can be defined by folding or unfolding these trees. When a tree is folded (or unfolded) we are just joining (or splitting) petals in the based partitions. The nice property about these labeled trees is that they only allow precisely those splits that will result in another vertex type. To generalize this idea, if $\tau$ is a vertex type in $\Wh_\G$ for an arbitrary $\G$, we say that a petal $I$ in a based partition $\tau_u$ of $\tau$ can be split when $\tau'_u$, the  based partition that results from the split,   is still compatible with every other based partition $\tau_v$ of $\tau$.
\end{nota}
 \newpage
\begin{defi}
    A vertex $\tau$ in $\Wh_n$ is \textit{essential} if for each based partition $\tau_j$ of $\tau$ we have that:
    \begin{itemize}
        \item[1)] If $x_j=x_1$, the only petal that  can be split to yield a new vertex type is the one containing $x_2$.
        \item[2)] If $x_j\neq x_1$, the only petal that can be split to yield a new vertex type  is the one containing $x_1$.
    \end{itemize}
\end{defi}

\section{Choosing the basis for $\POut(\A)$}\label{sec:basis}

From now on we fix an  order in the elements in $V(\G)$, and denote them by $v_1,...,v_n$. Sometimes, for simplicity and in the aim of a clear notation, we will just identify $v_i$ with $i$; in this sense, we will often put $\tau_i$ to denote the based partition $\tau_{v_i}$, or $C^j_I$ for the partial conjugation $C^{v_j}_I$.
Let $\tau$ be a vertex type in MM$_\G$ and $\tau_i$ a based partition of $\tau$. The product of the partial conjugations associated to each petal of $\tau_i$ gives the inner automorphism associated to $v_i$, so it is trivial in the outer group. We need to have a way to choose families of partial conjugations  associated to  $\tau_i$ which are independent.

\begin{defi}[Canonical automorphism]
Let $\tau$ be a vertex type in Wh$_\G$ and $\tau_j$ a based partition of $\tau$. For each petal $I$  in $\tau_j$  is  the  associated partial conjugation is $C_I^j$. Let $k$ be the smallest label in $\G-\st(v_j)$. We say that $C_I^j$  is a {\sl canonical automorphism} if $v_k\notin I$. The image of $C_I^j$ in $\POut(A_\G)$, which we will denote by $\bar{C}_I^j$,  is called a {\sl canonical outer-automorphism}.
\end{defi}

It is easy to check that, in the case of the free group, these are precisely the one-two automorphisms in Definition \ref{defi:one-two_aut}.

\begin{ej}
    Let $\G$ be as in Figure \ref{fig:gamma}, and  $$\tau_2=\{\{2\},\{4\},\{5,6,7,8,9,10,11\}\}$$  the only non-trivial based partition associated to $v_2\in\G$. The smallest label in $\G-\st(v_2)$ is 4, so the only canonical automorphism associated to $v_2$ is the partial conjugation $C_I^2$ where $I=\{5,6,7,8,9,10,11\}$. \label{ej:canonical}
\end{ej}

\begin{figure}
   \begin{tikzpicture}
       \centering

    \node[shape=circle,draw=black, fill=black, label={above: 2}, scale=0.5] (1) at (-6,-1.5) {};
    \node[shape=circle,draw=black, fill=black, label={above: 3}, scale=0.5] (2) at (-5,-1.5) {};
    \node[shape=circle,draw=black, fill=black, label={above: 4}, scale=0.5] (3) at (-4,-1.5) {};
    \node[shape=circle,draw=black, fill=black, label={above: 5}, scale=0.5] (4) at (-3,-1.5) {};
     \node[shape=circle,draw=black, fill=black, label={above: 6}, scale=0.5] (5) at (-2,-1.5) {};
    \node[shape=circle,draw=black, fill=black, label={above left: 8}, scale=0.5] (6) at (-4,-0.5) {};
    \node[shape=circle,draw=black, fill=black, label={above left: 9}, scale=0.5] (7) at (-4,0.5) {};
    \node[shape=circle,draw=black, fill=black, label={ above: 7}, scale=0.5] (8) at (-1,-1.5) {};
   
    \node[shape=circle,draw=black, fill=black, label={below left: 10}, scale=0.5] (9) at (-4,-2.5) {};
    \node[shape=circle,draw=black, fill=black, label={below left: 11}, scale=0.5] (10) at (-4,-3.5) {};
    \node[shape=circle,draw=black, fill=black, label={ above: 1}, scale=0.5] (11) at (-7,-1.5) {};

    \path (2) edge (1);
   \path (2) edge (9);
    \path (2) edge (6);
   
    \path (4) edge (5);
    \path (4) edge (9);
    \path (4) edge (6);
    \path (7) edge (6);
    \path (5) edge (8);
    \path (10) edge (9);
    \path (1) edge (11);
 
   \end{tikzpicture}
   
    \caption{A graph $\G$ for  Example \ref{ej:canonical}} \label{fig:gamma}
\end{figure}
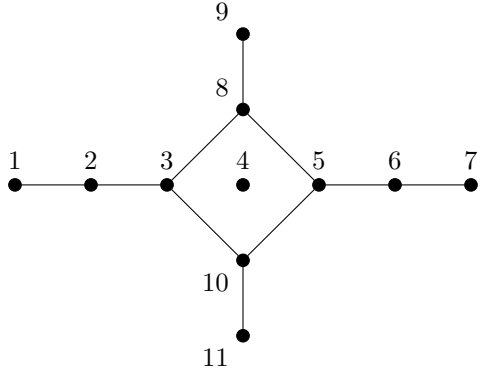

\begin{defi}[Canonical basis]
    For a given vertex type $\tau\in\MM_\G$, we define the {\sl canonical basis} as the collection of all the canonical outer-automorphisms $\bar{C}^j_I\in$ Stab$(\tau)$, i.e., all the $\bar{C}^j_I$ such that
\begin{itemize}
    \item [1)] $I$ is a petal in a based partition $\tau_j$ of $\tau$.
    \item[2)] $\bar{C}^j_I$ is a canonical outer-automorphism.
\end{itemize}
     We denote it by $\B(\tau)$. If $\sigma$ is a simplex in $|\Wh_\G|$ corresponding to a chain $\tau^0<\cdots<\tau^k$ then the stabilizer of $\sigma$ is the stabilizer of $\tau^0$ so we define $\B(\sigma):=\B(\tau^0)$. 
\end{defi}

Following the explanation in the Remark after Definition \ref{def:MM}, we have that

\begin{lem}
    The set $\B(\tau)$ is a minimal generating set for Stab$(\tau)$. \label{lem:baseGen}
\end{lem}

\begin{defi}[Support]
    Let $\au$ be a family of canonical outer-automorphisms, then we define its {\sl support} as
    $$
\textrm{supp}(\au):=\{\sigma\in |\Wh_\G|\ |\ \au\subseteq \B(\sigma)\}.
    $$
\end{defi}

\begin{nota}
     $\supp(\au)$ is not necessarily a subcomplex of $|\Wh_\G|$.
\end{nota}

\begin{defi}[Compatibility]
    A collection $\au$ of canonical outer-automorphisms is called {\sl compatible} when $\supp(\au)\neq \emptyset$.
\end{defi}

\begin{ej}
    It is possible for a canonical outer-automorphism $\bar{C}_{I}^j$ to be in Stab$(\tau)$ but not in $\B(\tau)$ for some vertex type $\tau$. For example, let us take $F_5$, the free group  of rank five. In the associated Whitehead poset $\Wh_5$, let $\tau$ be the vertex type where all  based partitions are trivial except for $\tau_3=\{\{3\},\{1,2\},\{4,5\}\}$; and let $\tau'$ be the vertex type where all  based partitions are trivial except for $\tau'_3=\{\{3\},\{1,2\},\{4\},\{5\}\}$, so $\tau<\tau'$. 
    
    We have that $\bar{C}_{\{4,5\}}^3$ is a canonical outer-automorphism since $v_1$ does not get conjugated. We note that $\bar{C}_{\{4,5\}}^3\in\B(\tau)$ since it is canonical and $\{4,5\}$ corresponds to a petal in the based partition $\tau_3$ of $\tau$; in fact, $\B(\tau)=\{\bar{C}_{\{4,5\}}^3\}$. Now, $\tau'$ has been obtained from $\tau$ by splitting the petal $\{4,5\}$, so we have that $\B(\tau')=\{\bar{C}_{\{4\}}^3,\bar{C}_{\{5\}}^3\}$. It is obvious that $\bar{C}_{\{4,5\}}^3\in\Stab(\tau')$, since it is the product of the two elements in the basis. However, $\bar{C}_{\{4,5\}}^3\notin\B(\tau')$, since $I=\{4,5\}$ is not associated to a single petal in the based partition $\tau'_3$ of $\tau'$.

   This also gives us an example of $\supp(\au)$ not being a subcomplex in $|\Wh_\G|$. Let $\au=\{\bar{C}_{\{4,5\}}^3\}$; if we take the simplex $\sigma$ given by the chain $\tau<\tau'$, we have that $\sigma\in\supp(\au)$ since $\B(\sigma)=\B(\tau)$. However, $\tau'\notin \supp(\au)$, so it is not a subcomplex. \label{ej:suppNotSub}
\end{ej}

\begin{defi}[Peripheral$(\au)$, $\mathcal{C}(\au)$]
    Let $\au$ be a family of canonical outer-automorphisms, then we define the following complexes
    $$
\peri(\au)=\{\sigma\in |\Wh_\G|\ |\ \au\subseteq \Stab(\sigma), \au\nsubseteq \B(\sigma)\}
    $$
    and 
    $$
\mathcal{C}(\au)=\{\sigma\in |\Wh_\G|\ |\ \au\subseteq \Stab(\sigma)\}.
    $$
\end{defi}

\begin{ej}
    Following with   Example \ref{ej:suppNotSub}, let $\au=\{\bar{C}_{\{4,5\}}^3\}$. We have that $\tau$ supports $\au$ since $\au\subseteq\B(\tau)$, so $\tau\in\supp(\au)$. On the contrary, $\bar{C}_{\{4,5\}}^3\in$ Stab$(\tau')$, but $\au$ is not supported by $\tau'$ since $\bar{C}_{\{4,5\}}^3\notin\B(\tau')$, so we have that $\tau'\in\peri(\au)$.
\end{ej}

\begin{defi}[Cone point of $\au$]
    For any compatible set $\au$  there is a unique vertex type $\tau$ such that $\B(\tau)=\au$. This vertex type is called the {\sl cone point} of $\au$ and denoted by $\tau(\au)$. \label{def:conePoint}
\end{defi}

Following the definition, we have that if $\bar{C}_{I_1}^i,...,\bar{C}_{I_k}^i$ is the set of canonical outer-automorphisms of $\au$ that conjugate by $v_i$, then the based partition associated to $v_i$ in $\tau(\au)$ will be $\tau(\au)_i=\{\{v_i\},I_1,...,I_k,M\}$; where $M=\G-(\st(v_i)\cup I_1\cup\cdots\cup I_k)$, and $M$ contains the minimal element of $\G-\st(v_i)$.
It is easy to check that for all vertex types $\tau$ the cone point of its canonical basis is $\tau$ itself; i.e., $\tau(\B(\tau))=\tau$. And for all compatible collections $\au$, the canonical basis of its cone point is $\au$ itself; i.e., $\B(\tau(\au))=\au$. As a consequence, the cone point $\tau(\au)$ lies in  $\suppa$.

The next Lemma is a direct consequence of Lemma \ref{lem:ordercarrier}.

\begin{lem}
    If $\au$ is a collection of compatible canonical outer-automorphisms, then  $\au\subseteq$ Stab$(\tau')$ if and only if $\tau(\au)\leq \tau'$.  \label{lem:ordenNuevo}
\end{lem}

\begin{lem}
    Both $\peri(\au)$ and $\mathcal{C}(\au)$ are subcomplexes of $\Wh_\G$.
\end{lem}

\begin{proof}
    Let $\sigma \in\peria$ with $\sigma=\tau^0<\cdots <\tau^k$. We have to check that every vertex type $\tau^i$ is also in $\peria$. By definition of $\peria$, we have that $\au\subseteq\B(\sigma)$ and $\au\nsubseteq\B(\sigma)$, which means that $\au\subseteq\Stab(\tau^0)$ and $\au\nsubseteq\B(\tau^0)$. We have to check that $\au\subseteq\Stab(\tau^i)$ and $\au\nsubseteq\B(\tau^i)$ for every $\tau^i$. Since $\au\subseteq\Stab(\tau^0)$ and $\tau^0<\tau^i$, by Lemma \ref{lem:ordercarrier}, we have that $\au\subseteq\Stab(\tau^i)$. 
    
    Now, since $\au\nsubseteq\B(\tau^0)$, there must be a canonical outer-automorphism $\bar{C}^j_I\in\au$ such that $\bar{C}^j_I=\bar{C}^j_{I_1}\cdots\bar{C}^j_{I_k}$, where $k>1$ and all the $\bar{C}^j_{I_i}\in\B(\tau^0)$, i.e, all the $I_i$ are petals in the based partition $\tau^0_j$ of $\tau^0$. As $\tau^0<\tau^i$, all the petals in the based partition  $\tau^i_j$ are either the same  or a result of splitting some petals in $\tau^0_j$. As a consequence, we have that $\bar{C}^j_{I}\notin\B(\tau^i)$, because in other case $I=I_1\cup\cdots\cup I_k$ would be a petal in the based partition $\tau^i_j$ of $\tau^i$, which is a contradiction.

    Applying Lemma \ref{lem:ordercarrier}, it is obvious that $\ca$ is a subcomplex.
\end{proof}
\begin{lem}
     Let $\au$ be a family of canonical outer-automorphisms, then $\ca=\{\sigma\in |\Wh_\G|\ |\ \tau(\au)\in\sigma\}$ and, as a consequence, $\ca$ is contractible (is a cone over $\tau(\au)$).
\end{lem}
\begin{proof}
    It follows from Lemma \ref{lem:ordenNuevo} and the definition of $
\mathcal{C}(\au)$.
\end{proof}

Recall that, as explained in the Remark after Definition \ref{defi:one-two_aut}, for a vertex type $\tau$, we cannot split  petals arbitrarily if we want to obtain another vertex type. A petal $I$ in a based partition $\tau_i$ of $\tau$ might not admit any split because of two possible reasons:
\begin{itemize}
    \item[a)] The petal is formed by a single connected component in $\G-\st(v_i)$, so it cannot be divided anymore.
    \item[b)] There is another based partition $\tau_j$  in  $\tau$, such that $v_i,v_j$ have a shared component; and the splitting of the petal in $\tau_i$ containing the dominant component $D^j$ produces a crossing with  $\tau_j$, breaking the compatibility. 
\end{itemize}

We are now ready to generalize the idea of essential vertices:

\begin{defi}[Essential]
    A collection of compatible canonical outer-automorphisms $\au$ is {\sl essential} if $\tau=\tau(\au)$ satisfies that for each based partition $\tau_i$ the only petal we can split to get a new vertex type is the one containing the minimal element in $\G-$st$(v_i)$. In this case we say that $\tau$ is an {\sl essential} vertex type. 
    Collections of compatible canonical outer-automorphisms or vertex types that are not essential are called {\sl inessential}.
\end{defi}

Let us explain the previous definition: the canonical basis is defined using the partial conjugations associated to the petals that do not contain the minimal element. If we split a petal which does not contain the minimal element, then the canonical basis will change. If we have an essential collection $\au$, then for each $\tau'$ such that $\tau(\au)\leq\tau'$ we have that $\tau'$ supports $\au$, since the petals that defined the canonical outer-automorphisms in $\au$ cannot be divided. So, in a simple way, having an essential collection $\mathcal{A}$ ensures that if $\tau$ supports $\mathcal{A}$ and $\tau\leq \tau'$, then $\tau'$ also supports $\mathcal{A}$.

\begin{lem}
    Let $\au$ be a collection of compatible canonical outer-automorphisms. Then $\au$ is essential if and only if the complex $\peria$  is empty and if and only if $\supp(\au)$ is a subcomplex of $|\Wh_\G|$. In this case, moreover, $\C(\au)=\supp(\au)$ so $\supp(\au)$ is contractible. \label{lem:EssentialCharact}
\end{lem}
\begin{proof}
    Let $\au$ be essential, $\tau=\tau(\au)$ and $\bar{C}_{I}^j\in\au$. By definition of canonical and essential, the petal  $I$ in the based partition $\tau_j$ does not contain the minimal element, so it cannot be divided. Since every vertex $\tau'$ with $\tau<\tau'$ is obtained by splitting petals in $\tau$, we have that in the based partition $\tau'_j$ we will still have the petal  $I$ (since it cannot be divided), so $\bar{C}_{I}^j\in\B(\tau')$ and $\au\subset\B(\tau')$.

    For the converse, if $\au$ is inessential, there must be at least one petal $I$ in a based partition $\tau_j$  which can be split and does not contain the minimal element in $\G-$st$(v_j)$. So, if we divide that petal into $I_1, I_2,$ we will obtain a $\tau'>\tau$ such that $\bar{C}_{I}^j=\bar{C}_{I_1}^j\bar{C}_{I_2}^j\in$ Stab$(\tau')$ but $\bar{C}_{I}^j\notin\B(\tau')$, so $\tau'$ is in the $\peria$ complex, so it cannot be empty.

    For the second equivalence, as explained in the previous paragraph, $\au$ is essential if and only if when $\tau$ supports $\mathcal{A}$ and $\tau\leq \tau'$, then $\tau'$ also supports $\mathcal{A}$. As a consequence, if $\sigma=\tau^0<\cdots<\tau^k\in\supp(\au)$, then for any $\tau^i$ we have that $\tau^i\in\supp(\au)$.
\end{proof}

We now turn to the case when $\mathcal{A}$ is inessential. It makes sense to think that if we take an inessential vertex type and we keep splitting its petals, at some point we should get  an essential vertex type. We aim to do this process in an ordered way.

\begin{defi}[Worrisome petal]
    Let   $\tau$ be an inessential vertex type. We call a petal in a based partition $\tau_i$ of $\tau$  {\sl worrisome} if it can be split and it does not contain the minimal element in $\G-\st(v_i)$. \label{def:worri}
\end{defi}  

\begin{nota}
    Observe that if $\tau$ is inessential, there must be at least one worrisome petal; and if it is essential there is none. Further, we know that if $\tau=\tau(\au)$ is inessential, then $\peria$ is not empty, and now we have another characterization: let $\tau'$ be a  vertex type, then $\tau'\in\peria$ if and only if $\tau'>\tau$ and at least one of the worrisome petals in $\tau$ is split in $\tau'$.
\end{nota}

We will state some technical lemmas about worrisome petals:
\begin{lem}
     Let $\tau$ be an inessential vertex type and $\tau_i=\{\{v_i\}, P_1,...,P_k\}$, $\tau_j=\{\{v_j\},Q_1,...,Q_m\}$ two based partitions of $\tau$ such that $v_i$ and $v_j$ are not adjacent. Let us suppose that $P_1$ contains $D^j$ and $Q_1$ contains $D^i$, the respective dominant components. Then,  $P_1$ and $Q_1$ cannot both be worrisome. \label{lem:DomWorri}
\end{lem}
\begin{proof}
   We have two different cases:
\begin{itemize}
    \item [(i)] The minimal elements in $\G-\st(v_i)$ and $\G-\st(v_j)$ are the same: In this case the minimal element is either in the dominant component of one of them (and in a subordinate component of the other) or in a shared component. If it is in the dominant of one of them,  say $D^j$, then $P_1$ is not worrisome since it contains the minimal element and we have finished. If it is in a shared component, we have that the shared component must lie in the same petal as the dominant component in at least one of the two based partitions, if not a cross would occur between them, contradicting the fact that the based partitions are compatible; so that petal is not worrisome. 
    \item[(ii)] The minimal element in $\G-\st(v_i)$ is different from the one in $\G-\st(v_j)$: We can assume that the smallest one is the one in $\G-\st(v_i)$. This means that the vertex with the smallest label in $\G-\st(v_i)$ is in $\st(v_j)$, as a consequence, it is in  $D^j$, so $P_1$ is not worrisome. 
\end{itemize}

\end{proof}

\begin{lem}
    Let $\tau$ be an inessential vertex type and $\tau_i=\{\{v_i\},I, P_1,...,P_k\}$, $\tau_j=\{\{v_j\},J,Q_1,...,Q_m\}$ two based partitions of $\tau$, such that both $I,J$ are worrisome petals. Let $\tau'_i=\{\{v_i\},I_1,I_2, P_1,...,P_k\}$ and $\tau'_j=\{\{v_j\},J_1,J_2,Q_1,...,Q_m\}$ be  based partitions  obtained by splitting the worrisome petals in an admissible way. Then $\tau'_i$ and $\tau'_j$ are compatible. \label{lem:worrisomeNotImpede}
\end{lem}

\begin{proof}
    The case when $v_i$ and $v_j$ are adjacent is trivial, since based partitions associated to adjacent vertices are always compatible. We study the case when they are not adjacent. Since $\tau'_i$ is obtained by   an admissible split, it means that it is compatible with al the based partitions of $\tau$, in particular, it is compatible with $\tau_j$. By the same reason, $\tau'_j$ is compatible with $\tau_i$. By Lemma \ref{lem:DomWorri}, we have that at least one of the worrisome petals does not contain the dominant component. 
    
    By symmetry we assume that $I$ is the one not containing $D^j$. Let us suppose that $\tau'_i$ is not compatible with $\tau'_j$, i.e., there is crossing between them.  By Lemma 5.11 in \cite{ArMar}, if $\tau_i$ is compatible with $\tau'_j$ but $\tau'_i$ is not, then the petal containing the dominant component in $\tau_i$ is split, but this is a contradiction, since the only petal that we have split is $I$, and $I$ does not contain the dominant component.

\end{proof}

In general, for a vertex type $\tau$, splitting a petal in a based partition $\tau_i$ of $\tau$ might imply that we cannot split some other petal in another based partition $\tau_j$ of $\tau$, because  crossings might appear. The previous lemma gives us a really important result for the rest of the section: it says that, if we have an inessential vertex type, splitting a worrisome petal does not impede another worrisome petal from being split.

\begin{lem}
    Let $\tau$ be an inessential vertex type in $\Wh_\G$. Then, there exists a unique essential vertex type $\widehat{\tau}>\tau$ obtained by splitting only the worrisome petals in $\tau$. \label{lem:essentialMaximum}
\end{lem}

\begin{proof}
    We will start by constructing $\widehat{\tau}$. We proceed in the following way: we take each based partition $\tau_i$ of $\tau$, and we split each worrisome petal in $\tau_i$ as much as possible, until there is no worrisome petal left. By Lemma \ref{lem:worrisomeNotImpede}, the previous procedure is independent of the order in which we split the petals, so we will always get to the vertex type $\widehat{\tau}$, which is essential, since it has no worrisome petal. 
\end{proof}

\begin{defi}[Essential cover]
We define the {\sl essential cover} of  an inessential vertex type $\tau$ as the unique essential vertex type $\widehat{\tau}$ obtained by splitting only the worrisome petals in $\tau$.
\end{defi}

\begin{lem}
    Let $\au$ be an inessential collection of compatible canonical outer-automorphisms. Then the complex $\peria$  is contractible. \label{lem:periContr}
\end{lem}

\begin{proof}

   Let $\tau=\tau(\au)$, and let $\mathfrak{t}>\tau$ be a vertex type obtained from $\tau$ by splitting at least one worrisome petal, and only worrisome petals.  Let $\Wh_{\geq\mathfrak{t}}$ be the subposet of $\Wh_\G$ consisting of vertex types greater than or equal to $\mathfrak{t}$. We have that the associated  subcomplexes $|\Wh_{\geq\mathfrak{t}}|$ cover $\peria$, since in every vertex type in $\peria$  at least one of the worrisome petals in $\tau$ has been split (recall the remark after Definition \ref{def:worri}).   
   We have that each $|\Wh_{\geq\mathfrak{t}}|$ is contractible, since it is a cone over $\mathfrak{t}$. Even more, if we have two different   $\mathfrak{t}_1$, $\mathfrak{t}_2$, by Lemma \ref{lem:worrisomeNotImpede}, they must be compatible;  so, by Lemma  5.6 in \cite{ArMar}, they have a least upper bound $\mathfrak{t}_3$ and it is clear that $|\Wh_{\geq\mathfrak{t}_1}|\cap |\Wh_{\geq\mathfrak{t}_2}|= |\Wh_{\geq\mathfrak{t}_3}|$. This means that we have covered $\peria$ by contractible complexes, whose intersection is also contractible. As a consequence we can apply the Quillen-McCord theorem (or Quillen fiber lemma) and deduce that $\peria$ is homotopic to the nerve of the covering.

   Now, let $\widehat{\tau}$ be the essential cover of $\tau$, we have that for every $\mathfrak{t}$, $\widehat{\tau}\geq\mathfrak{t}$, so $|\Wh_{\geq\widehat{\tau}}|\subseteq |\Wh_{\geq\mathfrak{t}}| $; as a consequence $|\Wh_{\geq\widehat{\tau}}|$ is a cone point in the nerve, hence the nerve is contractible.

\end{proof}
Since $\peria\subseteq\mathcal{C}(\au)$ and both are contractible we have:
\begin{lem}
    The relative homology groups $H_i(\mathcal{C}(\au),\peria)$ vanish for all $i\geq0$.
    \label{lem:HomoRel}
\end{lem}

\section{Computing $H^*(\POut(A_\G))$}\label{sec:PSO}

We want to compute the cohomology groups $H^q(\POut(A_\G))$ using the equivariant spectral sequence associated to the action of $\POut(A_\G)$ on $ \MM_\G$. Recall that if $G$ is  a group acting simplicially on a contractible complex $X$, the first page of this equivariant spectral
sequence (see, for example, Proposition 7.3 in  Chapter VII of \cite{Brown}) is

\begin{equation}
    E_1^{p,q}=\prod_{\substack{\sigma \in\left|W_\G\right|, \\ \operatorname{dim}(\sigma)=p}} H^q(\operatorname{Stab}(\sigma))
    \label{eq:spectralSeq}
\end{equation}

The differentials
on the $E_1$ page are a combination of restriction  and coboundary maps. 
We note that $\B(\tau)$,  the canonical basis  for the stabilizers of vertices, allows us to give a concrete description of the entries on the first page of this spectral sequence.

We know that the stabilizer of a cell $\sigma=\tau^0<\cdots<\tau^k$ is the stabilizer of $\tau^0$, so we can consider what happens in the case of a vertex  $\tau$. The stabilizers are free abelian, so  $H_1(\operatorname{Stab}(\tau), \mathbb{Z})$ is free abelian with generating set $\left\{\left[\bar{C}_{I}^j\right] \mid \bar{C}_{I}^j \in \mathcal{B}(\tau)\right\}$. As a consequence, $H^1(\operatorname{Stab}(\tau))$ is generated by the dual basis $\left\{\left[\bar{C}_{I}^j \right]^* \mid \bar{C}_{I}^j \in \mathcal{B}(\tau)\right\}$, we will denote  $\bar{\gamma}_I^j:=\left[\bar{C}_{I}^j\right]^*$.

For free abelian groups,  $H^*(\mathbb{Z}^n)$ is an exterior algebra generated by one-dimensional classes, so the set $\left\{\bar{\gamma}_I^j \mid \bar{C}_{I}^j \in \mathcal{B}(\tau)\right\}$ is a generating set for the cohomology of $\Stab(\sigma)$. As a consequence, the set of all products of $q$ distinct $\bar{\gamma}_I^j$ classes is a generating set for $H^q(\operatorname{Stab}(\sigma))$. Therefore, we may choose as generating set for $H^q(\operatorname{Stab}(\sigma))$  the collection of dual classes of all subsets of canonical outer-automorphisms $\mathcal{A} \subset \mathcal{B}(\tau)$ with $|\mathcal{A}|=q$.
Taking all of this into account, we see that in  $E_1^{p,q}$ we have a basis element for each compatible collection $\au$ of canonical outer-automorphisms with $|\au|=q$ and each $\sigma$ of dimension $p$ where $\mathcal{A}\subseteq\B(\sigma)$.

\subsection{The essential filtration}

We aim to prove that the spectral sequence collapses at $E_2$. To do that we will construct a filtration through essential and inessential vertex types.
We begin with a few useful remarks. Let   $G\curvearrowright X$ be a group action, with $X$ a contractible simplicial complex and  $X/G=W$ a finite fundamental domain. Let $F_{\bullet}(X)$ be the chain complex of $X$ and for any cell $\sigma\in X$ consider the stabilizer $\Stab(\sigma)$ and the induced module $\mathbb{Z}\uparrow^G_{\Stab(\sigma)} $. Then 
$$F_p(X)=\bigoplus_{\substack{\sigma\in W \\ |\sigma|=p}}\mathbb{Z}\uparrow^G_{\Stab(\sigma)}.$$
Now, for the homology equivariant spectral sequence we have:

$$
E^1_{p,q}=H_q(G,F_p(X))=H_q(G,\bigoplus_{\substack{\sigma\in W \\ |\sigma|=p}}\mathbb{Z}\uparrow^G_{\Stab(\sigma)})= \bigoplus_{\substack{\sigma\in W \\ |\sigma|=p}}H_q(G,\mathbb{Z}\uparrow^G_{\Stab(\sigma)})=\bigoplus_{\substack{\sigma\in W \\ |\sigma|=p}}H_q(\Stab(\sigma),\mathbb{Z}).
$$
In a similar way, taking into account that $W$ is finite

$$
E^{p,q}_1=H^q(G,\textrm{Hom}_\mathbb{Z}(F_p(X),\mathbb{Z}))=H^q(G,\bigoplus_{\substack{\sigma\in W \\ |\sigma|=p}}\textrm{ Coind }^G_{\Stab(\sigma)}\ \mathbb{Z})= \bigoplus_{\substack{\sigma\in W \\ |\sigma|=p}}H^q(G,\mathbb{Z}\uparrow^G_{\Stab(\sigma)})=\bigoplus_{\substack{\sigma\in W \\ |\sigma|=p}}H^q(\Stab(\sigma)).
$$
Since in our case all stabilizers are free abelian, we have that
$$
H^q(\Stab(\sigma))=\textrm{Hom}_\mathbb{Z}\left(H_q(\Stab(\sigma)),\mathbb{Z}\right).
$$
As a consequence, $E_1^{\bullet,q}$ can be obtained by dualizing $E^1_{\bullet,q}$, so if $E^1_{\bullet,q}$ has homology zero always except in degree zero where it is free abelian of rank $k$, the same will happen for $E_1^{\bullet,q}$. Therefore, for simplicity, we will work with the homology  spectral sequence. We want to construct a filtration to compute the homology of $E^1_{\bullet,q}$. Because of Lemma \ref{lem:essentialMaximum}, we know that for each inessential vertex $\tau$ we can find its essential cover $\widehat{\tau}$, an essential vertex  in $\peria$. 
\begin{defi}[$s(\au)$]
    Let $\au$ be a family of compatible canonical outer-automorphisms. If $\au$ is essential we put $s(\au)=0$, in other case let $\tau=\tau(\au)$ and $\widehat{\tau}$  its essential cover, then $s(\au)$ is the  distance between $\tau$ and $\widehat{\tau}$, i.e. the length of any unrefinable chain $\tau<\tau^1<\cdots<\tau^{s(\au)}=\widehat{\tau}$ . \label{def:lengthIness}
\end{defi}

We claim that there is a finite filtration
$$
F^0_{\bullet}\subset F^1_{\bullet} \subset \cdots \subset  E^1_{\bullet,q}
$$
  where $F_p^s=\bigoplus\mathbb{Z}_{\au}$, and the sum ranges over the families $\au\subset \B(\sigma)$ with $\sigma$ a $p$-cell such that $s(\au)\leq s$ and $|\au|=q$. Moreover
$$
F_\bullet^s/F^{s-1}_\bullet=\bigoplus_\au \textrm{ Chain complex of the pair }(\mathcal{C}(\au),\peri(\au))
$$
ranging over the set of  compatible families $\au$, with $|\au|=q$ and $s(\au)=s$; and 
$$
F^0_{\bullet}=\bigoplus_\au \textrm{ Chain complex of } \supp(\au)
$$
where now the sum is over the set of  compatible families $\au$ such that $|\au|=q$ and $\au$ essential.
To see it we have to check that the differential map $d^1$ of $E^1_{p,q}$ maps $F_\bullet^s$ to $F_\bullet^{s}$. We do this in our next lemma.

\begin{lem}
There is a finite filtration
$$
F^0_{\bullet}\subset F^1_{\bullet} \subset \cdots \subset  E^1_{\bullet,q}
$$
   with 
   $$F_p^s=\bigoplus\{\mathbb{Z}_{\au}\ |\ \au\subseteq\B(\sigma), \sigma\in|\Wh_\G| \textrm{ a }p\textrm{-cell}, |\au|=q, s(\au)\leq s\}$$ 
   such that $F_\bullet^s/F^{s-1}_\bullet$ is the sum of the chain complexes of pairs $(\mathcal{C}(\au),\peri(\au))$ and $
F^0_{\bullet}$ is the sum of chain complexes of $\suppa$ for $\au$ essential.

\end{lem}

\begin{proof}
The differential
$$
d^1:E^1_{p,q}\rightarrow E^1_{p-1,q}
$$
comes from 
$$d^1_\sigma :H_q(\Stab(\sigma),\mathbb{Z})\rightarrow \bigoplus H_q(\Stab(\sigma_i),\mathbb{Z})$$
where $\sigma=\tau^0<\cdots<\tau^p$ and $\sigma_i$ is the result of removing $\tau^i$ from $\sigma$. So we have two possible cases:

\begin{itemize}
    \item If $i\neq 0$: $\Stab(\sigma_i)=\Stab(\sigma)=\Stab(\tau^0)$.
    \item If $i= 0$: $\Stab(\sigma_0)=\Stab(\tau^1)$.
\end{itemize}

As a consequence, for $\mu_\au\in F_p^s$ corresponding to the $p$-cell $\sigma$ we have:
$$
d^1_\sigma(\mu_\au)=\sum_{i=0}^p(-1)^i cores_i(\mu_\au)
$$
where $cores$ denote the corestriction map $
H_q(\Stab(\sigma))\xrightarrow{cores}  H_q(\Stab(\sigma_i))
$. So   $cores_i(\mu_\au)\in F^s_p$ for $i\in\{1,...,p\}$, since they correspond to the same copy of $\au$, i.e., they are in $\C(\au)$. In the case when $i=0$ we distinguish two cases:

   \noindent{\textbf{ Case 1:}} If we have obtained $\tau^1$ from $\tau^0$ without splitting any worrisome petal, then $\au\subset\B(\tau^1)$, so $\tau^1\in\C(\au)$ and the image is $F^s_p$.
   
    \noindent{\textbf{ Case 2:}} If we have obtained $\tau^1$ from $\tau^0$ by splitting a worrisome petal $I$ into $I_1$ and $I_2$: 

    Let $S=\Stab(\tau), S_1=\Stab(\tau^1)$, so $S<S_1$. Let $\au\subseteq\B(\tau)$ such that $\bar{C}_I^i\in\au$, and $\au_1$, $\au_2$ be the sets formed by all the elements in $\au$ except for $\bar{C}_I^i$ which is replaced by $\bar{C}_{I_1}^i$ and $\bar{C}_{I_2}^i$ respectively. So $|\au_1|=|\au_2|=q$ and
$$
cores(\mu_\au)\in\mathbb{Z}\textrm{-span}\left\{\mu_{\au_j}\ | \ j\in\{1,2\}\right\}.
$$
    We take $\tau(\au)$ and $\tau(\au_1)$. All the based partitions are equal except for $\tau(\au)_i$ and $\tau(\au_1)_i$, which differ in two petals:
    $$
    \tau(\au)_i=\{\{i\},Q_1,...,Q_k,I,M\}
    $$
    and
    $$
    \tau(\au_1)_i=\{\{i\},Q_1,...,Q_k,I_1,M'\}
    $$
    where $M'=M\cup I_2$ and both $M$ and $M'$ contain the minimal element in $\G-\st(v_i)$ (recall Definition \ref{def:conePoint}). Since $M$ and $M'$ contain the minimal element, they cannot be worrisome; so the only possible difference between worrisome petals is the one between $I$ and $I_1$. We know that $I$ can be split at least once more than $I_1$, so the distance from $\tau(\au)$ to $\widehat{\tau}$ is bigger than the one from $\tau(\au_1)$ to $\widehat{\tau}_1$; therefore, $s(\au_1)<s(\au)$ and the image is in $F^s_p$. The case of $\tau(\au_2)$ is symmetric.
\end{proof}

As a consequence, all the previous lemmas together with Lemma \ref{lem:HomoRel} imply that 
$$
H_i(F_\bullet^l/F^{l-1}_\bullet)=0
$$
for all $i$, and we have, using also Lemma \ref{lem:EssentialCharact}:
\begin{lem}
    $$H_i(E^1_{\bullet,q})=H_i(F^0_\bullet)=\begin{cases}
        \mathbb{Z}^m, \ \ \ i=0 \\
        0, \ \ \ \ \ \ i>0
    \end{cases}
$$
where $m$ is the number of essential compatible families $\au$ with $|\au|=q$. \label{lem:homoE1}
\end{lem}

\subsection{The $E_2$ page}

By Lemma \ref{lem:homoE1}, we get  
\begin{prop}
    The non-zero entries on the $E_2$ page of the spectral sequence 
 are concentrated in the $(0,q)$-column, and the entries are free abelian groups of rank the number of essential compatible  families $\au$ with $|\au|=q$. \label{prop:column}
 \end{prop}

By definition, every essential set of compatible canonical basis elements $\au$ is associated to an essential  vertex type $\tau(\au)$, so we do also have:

\begin{cor}\label{cor:rank}
    Let $K_q$ be the number of rank $q$ essential vertex types in $\Wh_\G$, then
    $$
       H^q(\POut(A_\G))=\mathbb{Z}^{K_q}. 
    $$
\end{cor}
\begin{nota}
    The number $K_q$ can be computed algorithmically (see the Appendix \ref{appendix}).
\end{nota}

\begin{teo}
    The algebra $H^*(\POut(A_\G))$ is generated by the one-dimensional classes $\bar{\gamma}_I^j$ where  the $\bar{C}_{I}^j$ are canonical outer-automorphisms. \label{teo:oneDimPSO}

\end{teo}
\begin{proof}

By Lemma \ref{lem:baseGen},  $\POut(A_\G)$ is generated by $\left\{\bar{C}_{I}^j \mid C_{I}^j  \text{ is a canonical automorphism} \right\}$, its first homology is generated by the associated classes $\left[\bar{C}_{I}^j\right]$ and $H^1\left(\POut(A_\G)\right)$ is generated by the dual basis $\left\{\bar{\gamma}_I^j\right\}$. As we have seen in Proposition \ref{prop:column}, the $E_2$ page is concentrated in a single column and the groups in $E_2^{0, q}$ are free abelian.  The  generating set corresponds to products of elements of the form  $\bar{\gamma}_{K}^j=\sum_{I \in K} \bar{\gamma}_{I}^j$. These cohomology classes $\bar{\gamma}_{K}^j$ come from a generating set for the stabilizer of an essential vertex type, so  the partial conjugations $C_{I}^j$ are canonical automorphisms.
\end{proof}

Theorem A follows from Corollary \ref{cor:rank} and Theorem \ref{teo:oneDimPSO}.

\begin{ej}
    We will compute the cohomology groups $H^q(\POut(\A))$ for $$A_\G=\langle v_1, v_2, v_3, v_4, v_5\ |\ [v_1,v_2]=1\rangle$$  by computing the number of essential vertices in $\Wh_\G$.
    \begin{figure}[ht!]
    \centering
   \begin{tikzpicture}[
    vertex/.style={circle, draw, fill=black, inner sep=2.3pt},
    node distance=8cm
]
 \node at (-5,5.5) {$\G=$};

  \node[vertex, label={below: $v_1$}] (u1) at (-4.4,5.5) {};
    \node[vertex, label={below: $v_2$}] (u2) at (-3.4,5.5) {};
    \node[vertex, label={right: $v_4$}] (u3) at (-1.6,5.5) {};
    \node[vertex, label={above: $v_3$}] (u4) at (-2.4,6.2) {};
    \node[vertex, label={below: $v_5$}] (u5) at (-2.4,4.8) {};

 \draw (u1) -- (u2);

\end{tikzpicture}
 \caption{The graph $\G$}
 \label{fig:graphsEx}
\end{figure}
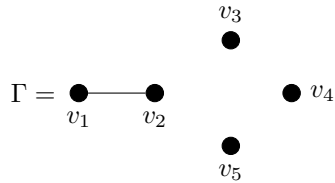
    As we can see, $\A$ is nearly a free group, as a consequence, $\Wh_\G$ will have some similarities with $\Wh_4$ (a representation of $\Wh_4$ can be found at Figure 2 in \cite{MacCammondMeier}); however, the adjacency between $v_1$ and $v_2$ results in their associated based partitions always being compatible, so the number of possible combinations increases. 

    The subposet of $\Wh_\G$ associated to vertex types where   based partitions associated to $v_1$ and $v_2$ are trivial is the same as the subposet of $\Wh_4$ where the based partitions associated to one of the vertices are always trivial, i.e, nine rank 1  and nine rank 2 vertex types.

    The subposet of $\Wh_\G$ associated to vertex types where  the only non-trivial based partitions are the ones associated to $v_1$  is the same as the subposet of $\Wh_4$ where the only non-trivial based partitions are the ones associated to a single vertex, i.e., three rank 1 vertex types and 1 rank 2 vertex type. The same occurs with the subposet  associated to vertex types where  the only non-trivial based partitions are the ones associated to $v_2$. Since $v_1$ and $v_2$ are adjacent, the subposet of $\Wh_\G$ associated to vertex types where  the only non-trivial based partitions are the ones associated to $v_1$ and $v_2$ is the direct product of their respective subposets. This direct product results in six rank 1, eleven rank 2, six rank 3 and one rank 4 vertex types.

    We are left with the possible combinations of $\{v_1,v_2\}$ with $\{v_3,v_3,v_5\}$. There are precisely two rank 2 vertex types with non-trivial based partitions at $v_1$ and $v_3$: if we split $v_4$ in one based partition we have to split $v_5$ in the other one, and vice versa. The same happens for any pair $(v_i,v_j)$ with $i\in\{1,2\}$ and $j\in\{3,4,5\}$. As a consequence, we  have twelve rank 2 vertex types associated to these pairs. There will also be two rank 3 vertices associated to the triple $(v_1,v_2,v_3)$: if we have split $v_4$ in the based partition associated to $v_3$, then we can split $v_5$ in  the based partitions associated to both $v_1$ and $v_2$ and the other way around. In the same way, we have two rank 3 vertices associated to the triples $(v_1,v_2,v_4)$ and $(v_1,v_2,v_5)$. 

    Once we described all the possible vertex types, we have to see which of them are essential. For each $v_i$, we have that $\G-\st(v_i)$ has three connected components, this means that a vertex type $\tau$ will be inessential if and only if it has a based partition $\tau_i=\{\{v_i\}, \{v_j\},\{v_k,v_l\}\}$ such that $v_j$ is the minimal element in $\G-\st(v_i)$ and the petal $\{v_k,v_l\}$ can be split.
    As a consequence, between the fifteen rank 1 vertex types, ten of them are essential (the ones where the minimal element has not been split). 
    
    For the rank 2  vertex types, the nine associated to $\{v_3,v_4,v_5\}$ are essential, since they are maximal (any other split would produce a crossing). Between the eleven associated to $\{v_1,v_2\}$, six of them are essential; the two vertex types where one of the based partitions is totally split and the rest are trivial, and the ones where the minimal element has not been split in any of the based partitions. All the twelve rank 2 vertex types associated to mixed pairs $(v_i,v_j)$ with $i\in\{1,2\}$ and $j\in\{3,4,5\}$ are essential, since the only based partition which admits a split is the one associated to $v_2$ in case $i=1$ or vice versa, and that based partition is trivial.

    For the rank 3  vertex types, the six associated to triples $(v_1,v_2,v_j)$ with $j\in\{3,4,5\}$ are essential since they are maximal. It remains to check the six vertex types associated to the pair $(v_1,v_2)$. In these vertex types one of the based partitions is totally split and  the other one will be $\tau_i=\{\{v_i\}, \{v_j\},\{v_k,v_l\}\}$. The essential ones are the ones where $v_j$ is not the minimal element, so there are four of them.

    In summary, we have the next information:
    \begin{table}[h]
\centering
\begin{tabular}{|c|c|c|}
\hline
\textbf{Rank} & \textbf{Vertex types} & \textbf{Essential} \\
\hline
0 & 1 & 1 \\
1 & 15 & 10 \\
2 & 32 & 27 \\
3 & 12 & 10 \\
4 & 1 & 1 \\
\hline
\end{tabular}
\caption{$\Gamma$-Whitehead poset}
\label{tab:whitehead_poset}
\end{table}

And, as a consequence
$$
H^q(\POut(\A))=\begin{cases}
    \mathbb{Z}, \textrm{ if } q=0 \\
    \mathbb{Z}^{10}, \textrm{ if } q=1 \\
    \mathbb{Z}^{27}, \textrm{ if } q=2 \\
    \mathbb{Z}^{10}, \textrm{ if } q=3 \\
    \mathbb{Z}, \textrm{ if } q=4 \\
    0 \textrm{ if } q\geq 5
\end{cases}
$$
Observe that even if $\A$ looks close to $F_4$, the differences in the homology of the symmetric outer-automorphisms group are remarkable:  by the formula given in Lemma 5.2 in \cite{JMcCM}, we get 
$$
H^q(\POut(F_4))=\begin{cases}
    \mathbb{Z}, \textrm{ if } q=0 \\
    \mathbb{Z}^{8}, \textrm{ if } q=1 \\
    \mathbb{Z}^{16}, \textrm{ if } q=2 \\
    0 \textrm{ if } q\geq 3 
\end{cases}
$$
In Appendix \ref{appendix}, we give a brief explanation of  an algorithm that computes not just the number of essential vertex types in $\Wh_\G$, but also the $n$-simplices in $|\Wh_\G|$ and, can be used to compute the $E^1$ page of the equivariant spectral sequence for the $\POut(A_\G)$ action on $MM_\G$. The algorithm can be implemented in Python and be used to check the previous example; though we encourage the inquisitive reader to try to do it by hand (a couple of paper sheets will be probably needed).
\end{ej}

\section{Computing $H^*(\PAut(A_\G))$}\label{sec:PSA}
\subsection{Some previous results}
We begin with some well known results about homology and cohomology of right-angled Artin groups, which we will be necessary later, and follow from the fact that the Salvetti complex associated to a graph $\G$, is a classifying space for the RAAG $A_\G$ (see \cite{Kob}). 

\begin{prop}
    Let $R$ be a commutative ring and $A_\G$ a RAAG, then
    $$
        H_k(A_\G,R)\cong R^{N_k}
    $$
    where $N_k$ is the number of $k$-cliques in $\G$. The cohomology group $H^k(A_\G,R)$ has the same rank.\label{prop:HomRaag}
\end{prop}

\begin{prop}
    Let $R$ be a commutative ring and $A_\G$ a RAAG. Let $n$ be the number of vertices of $\G$ and $\{e^1,...,e^n\}$ the standard basis of $H^1(\A,R)$. Then the cohomology ring $H^\ast(A_\G,R)$ is the exterior algebra $\bigwedge\left(R\right)^n$  modulo the relations $e^i\wedge e^j=0$ whenever the  vertices in $\G$ associated to $i$ and $j$ are not adjacent.  \label{prop:1clases}
\end{prop}

\noindent Note that, in particular, the previous proposition implies that $H^\ast(A_\G,R)$ is generated in degree 1.

\subsection{Applying the Leray-Hirsch Theorem}
We are assuming that $Z(\A)$ is trivial (see the remark after Theorem \ref{teo:partialGen}), so we have a  short exact sequence:
$$
1\to A_\G\to \PAut(A_\G)\to \POut(A_\G) \to 1
$$

\noindent which allows us to apply the Lyndon-Hochschild-Serre spectral sequence. As cd$(A_\G)=k$, where $k$ is the size of the biggest clique in $\G$, then this spectral sequence is concentrated in the first $k+1$ rows of the first quadrant. 

Exactly as in Section 6 in \cite{JMcCM}, there is a fibration which allows us to apply the Leray-Hirsch Theorem. This theorem implies that the differential $d^2$ of the spectral sequence is trivial, so the sequence collapses at page 2. This second page is  given by
$$E_2^{*, q}=H^*(\POut(A_\G), H^q(A_\G))=H^*(\POut(A_\G),\Z^{t_q})$$
where $t_q$ is the number of $q$-cliques in $\G$.  Following the argument of \cite{JMcCM}, we only need to check $H^k(\A)$ is $\Z$-free, which is given by Proposition \ref{prop:HomRaag} and the statement of the next lemma.

\begin{lem}
The map 
$$
\iota^*: H^*(\PAut(A_\G)) \rightarrow H^*( A_\G)
$$
is surjective, where $\iota:  A_\G\rightarrow  \PAut(A_\G)$ is the inclusion. \label{lem:tecnic}

\end{lem}

\begin{proof}
    We claim first that the map

        $$
H^1(\PAut(A_\G)) \rightarrow H^1(\A)
$$
is surjective. To see it, note that the injection $A_\G \hookrightarrow \PAut(A_\G)$ sends  each basis element $v_j$  to the symmetric automorphism

$$
C_{[n] \backslash\{\mathrm{st}(v_j)\}}^j=C_{I_1}^j \cdot C_{I_2}^j\cdots C_{I_k}^j
$$
where the $I_1,...,I_k$ are connected components in $\G-\mathrm{st}(v_j)$.  $H_1\left(\PAut(A_\G)\right)$ and $H_1\left(\A\right)$  are  free abelian groups with the same generating set  as   $\PAut(A_\G)$ and $A_\G$  respectively. The partial conjugations form a generating set for $\PAut(A_\G)$ that gives rise to a minimal generating set for the abelianization, so the map

$$
H_1\left(\A\right) \rightarrow H_1\left(\PAut(A_\G)\right) 
$$

\noindent is injective and the claim follows. As, by Proposition \ref{prop:1clases}, $H^*(\A)$ is generated in degree one, we see that the map

$$
H^*(\PAut(A_\G)) \rightarrow H^*(\A)
$$

\noindent is also surjective.

\end{proof}

\begin{cor} Let $K_i$ be the number of k essential vertex types of rank $i$ (generators of $H^i(\POut(A_\G) )$) and   $N_j$  the number of $j$-cliques in $\G$  (generators of $H^j(\A)$, with $N_0=1$).  Then

 $$
 H^q(\PAut(A_\G))\cong \bigoplus_{i+j=q} H^i(\POut(A_\G)) \otimes H^j(\A) 
 $$ 
 $$
\cong \bigoplus_{i+j=q} H^i(\POut(A_\G),H^j(A_\G)  ) \cong \bigoplus_{i+j=q} \mathbb{Z}^{K_i\cdot N_j}
$$

\end{cor}
\begin{proof}
    The result follows directly from    Corollary \ref{cor:rank}, Proposition \ref{prop:HomRaag} and the Leray-Hirsch Theorem. 
\end{proof}

So, with the same argument as in  Corollary 6.4 in \cite{JMcCM}, we get:
\begin{cor}
   The ring $H^*(\PAut(A_\G))$ is generated in degree 1. \label{cor:onedimPAut}
\end{cor}
This concludes the proof of Theorem B.

\section{The cohomology ring $H^*(\PAut(\A))$}
\label{sec:ring}

In this section we will give a ring presentation that we conjecture gives the   cohomology ring  $H^*(\PAut(\A))$. We also prove some partial results that provide some evidence for this conjecture.
Let $i,j$ be two vertices which are not adjacent in $\G$, let $D^j, D^i$ be the respective dominant components of $\G-\st(v_i)$, $\G-\st(v_j)$ and $\mathcal{C}$ a shared component. As in Section \ref{sec:PSO}, we denote by  $\gamma_{A}^i$  the dual class $[C_{A}^i]^*$. We want to prove that the next two relations  hold in $H^*(\PAut(\A))$:

\begin{itemize}
    \item[(1)] $\gamma_{D^j}^i\wedge \gamma_{D^i}^j=0$ 
    \item[(2)] $\gamma_{\mathcal{C}}^i\wedge \gamma_{\mathcal{C}}^j=\gamma_{\mathcal{C}}^i\wedge \gamma_{D^i}^j+ \gamma_{D^j}^i\wedge\gamma_{\mathcal{C}}^j$ 
\end{itemize}

\begin{prop} 
   With the previous notation, relation $(1)$ holds in $H^*(\PAut(A_\G))$.\label{prop:rel1}
\end{prop}

\begin{proof}
    We will prove the relation through an epimorphism from $\PAut(A_\G)$ to $F_2$, the free group generated by $\{x,y\}$. We define 
    $$
\varphi_1: \PAut(A_\G) \rightarrow F_2
    $$
    induced by 
    $$
    C_{D^j}^i\mapsto x, \quad
   C_{D^i}^j\mapsto y, \quad
    \alpha\mapsto1 \textrm{ for } \alpha \textrm{ any other generator. } 
    $$

 \noindent Taking into account the presentation of $\PAut(\A)$ given in Theorem \ref{teo:presentation}, we see that there is no possible commutation between $C_{D^j}^i$ and $C_{D^i}^j$ so this map extends to a well defined group homomorphism. The map $\varphi_1$ induces  a homomorphism $\overline{\varphi}_1^*$ between the first cohomology  groups:
$$
\overline{\varphi}_1^*:H^1(F_2) \rightarrow H^1(\PAut(A_\G)).
$$

\noindent With a suitable order in the generators of $\PAut(A_\G)$, $\overline{\varphi}_1$ can be represented by the following matrix:
$$
\begin{pmatrix}
1 & 0 & 0 & \cdots & 0 \\
0 & 1 & 0 & \cdots & 0
\end{pmatrix}
$$

\noindent so, dualizing, $\overline{\varphi}_1^*$ corresponds to the transpose,
which means $\overline{\varphi}_1^*(x^*)=\gamma_{D^j}^i$ and $\overline{\varphi}_1^*(y^*)=\gamma_{D^i}^j$.
In $H^1(F_2)$, we have that $x^*\wedge y^*=0$; and the naturality of the cup product with respect to group homomorphisms implies
$$
\varphi_1^*(x^*\wedge y^*)=\varphi_1^*(x^*)\wedge\varphi_1^*(y^*)=\gamma_{D^j}^i\wedge \gamma_{D^i}^j=0.
$$
\end{proof}

\begin{prop}
   With the notation established before, relation $(2)$ also holds in $H^*(\PAut(A_\G))$. \label{prop:rel2}
\end{prop}
\begin{proof}
    We argue as in Proposition \ref{prop:rel1}. We define 
    $$
\varphi_2: \PAut(A_\G) \rightarrow F_2
    $$
    induced by 
    $$
    C_{D^j}^i\mapsto x, \quad
   C_{D^i}^j\mapsto y, \quad
   C_{\mathcal{C}}^i\mapsto x^{-1}, \quad
   C_{\mathcal{C}}^j\mapsto y^{-1}, \quad
    \alpha\mapsto1 \textrm{ for } \alpha \textrm{ any other generator. } 
    $$

\noindent We will  verify again that this extends to a well defined group homomorphism. Taking into account  the presentation given in Theorem \ref{teo:presentation}, we see that the only relations that involve the preimages of $x^{\pm1}$ and $ y^{\pm1}$ are $[C_{\mathcal{C}}^iC_{D^j}^i, C_{\C}^j]=1$ and $[C_{\mathcal{C}}^jC_{D^i}^j, C_{\mathcal{C}}^i]=1$. It is enough to check that $\varphi_2(C_{\mathcal{C}}^iC_{D^j}^i C_{\mathcal{C}}^j)=\varphi_2(C_{\mathcal{C}}^jC_{\mathcal{C}}^iC_{D^j}^i )$ (the other relation will follow by symmetry):
$$
\varphi_2(C_{\mathcal{C}}^iC_{D^j}^i C_{\mathcal{C}}^j)=x^{-1}xy^{-1}=y^{-1}=y^{-1}x^{-1}x=\varphi_2(C_{\mathcal{C}}^jC_{\mathcal{C}}^iC_{D^j}^i ).
$$

\noindent Again, with a suitable order in the generators of $\PAut(A_\G)$, $\overline{\varphi}_2$ can be represented by the following matrix:
$$
\begin{pmatrix}
1 & 0 & -1 & 0 & 0&\cdots & 0 \\
0 & 1 & 0 & -1 & 0&\cdots & 0
\end{pmatrix}
$$
so  $\overline{\varphi}_2^*$, the induced  homomorphism between the first cohomology groups, corresponds to the transpose, and we get
$$\overline{\varphi}_2^*(x^*)=\gamma_{D^j}^i-\gamma_\C^i \ \ \textrm{ and }  \ \ \overline{\varphi}_2^*(y^*)=\gamma_{D^i}^j-\gamma_\C^j.$$

\noindent By the properties of the cup product  
$$
\overline{\varphi}_2^*(x^*\wedge y^*)=\overline{\varphi}_2^*(x^*)\wedge\overline{\varphi}_2^*(y^*)=(\gamma_{D^j}^i-\gamma_\C^i)\wedge (\gamma_{D^i}^j-\gamma_\C^j)=0
$$
and reordering and applying Proposition \ref{prop:rel1}, we obtain 
$$\gamma_{\mathcal{C}}^i\wedge \gamma_{\mathcal{C}}^j=\gamma_{\C}^i\wedge \gamma_{D^i}^j+ \gamma_{D^j}^i\wedge\gamma_{\mathcal{C}}^j.$$
\end{proof}

\begin{defi}\label{def:ringR}
    We define $R$ as the graded commutative ring generated by all the $\gamma_A^i$, where $1\leq i\leq n$ and $A$ ranges over the connected components in $\G-\st(v_i)$, with the relations 
    \begin{itemize}
        \item[(1)] $\gamma_A^i\wedge\gamma_A^i=0$
        \item[(2)] $\gamma_A^i\wedge\gamma_B^j=-\gamma_B^j\wedge\gamma_A^i$
        \item[(3)] $\gamma_{D^j}^i\wedge \gamma_{D^i}^j=0$ 
        \item[(4)] $\gamma_{\mathcal{C}}^i\wedge \gamma_{\mathcal{C}}^j=\gamma_{\mathcal{C}}^i\wedge \gamma_{D^i}^j+ \gamma_{D^j}^i\wedge\gamma_{\mathcal{C}}^j$ for each shared component $\C$ of $v_i,v_j$ .
    \end{itemize}
    
We will denote by $R_q$ the term of degree $q$ in R.
\end{defi}

By Propositions \ref{prop:rel1} and \ref{prop:rel2}, we have:

\begin{cor}
    Let $R$ be the ring of Definition \ref{def:ringR}. Then there exists a ring epimorphism
    $$
    \phi:  R  \twoheadrightarrow H^*(\PAut(A_\G)).
    $$

  so
$$
\textrm{dim}(R_q)\geq \textrm{dim}(H^q(\PAut(A_\G)).
    $$    
    \label{cor:H2_1}
\end{cor}

We want to prove now that $\phi$ is an isomorphism in degree 2. To do that,  we will  construct an epimorphism from the standard $\mathbb{Z}$-basis of
\begin{equation}
     H^2(\PAut(A_\G))\cong \bigoplus_{i+j=2} \mathbb{Z}^{K_i\cdot N_j}\cong \mathbb{Z}^{K_2\cdot N_0}\oplus\mathbb{Z}^{K_1\cdot N_1}\oplus\mathbb{Z}^{K_0\cdot N_2}
\end{equation}
to a $\mathbb{Z}$-basis of $R_2$. In the formula, we have $N_0=1$, and $N_1$, $N_2$ are the number of vertices and edges in $\G$ respectively.

\begin{defi}
    We will denote by $\B_1$  the standard $\mathbb{Z}$-basis of $H^2(\PAut(\A))$, which is bijective to the union of the following sets:
    \begin{itemize}
        \item Type 1 elements: the set of all rank 2 essential vertex types.
        \item Type 2 elements:  the set of tall pairs $(\tau,v_j)$ where $\tau$ is a rank 1 essential vertex and $v_j$ a vertex in $\G$.
        \item Type 3 elements: the set of all edges in $\G$. 
    \end{itemize}
\end{defi}

\begin{defi}
    Let $\B_2$ be the family of elements the  of $R_2$ consisting of those $\gamma_A^i\gamma_B^j$ such that $A\subset \G-\st(v_i), B\subset\G-\st(v_j)$ are connected components,  if $v_i=v_j$ then $A\neq B$ and when $v_i\notin\st(v_j)$
    \begin{itemize}
        \item $A,B$ cannot be shared and equal.
        \item $A,B$ cannot be both the respective dominant components.
    \end{itemize}
\end{defi}

\begin{lem}
    $\B_2$ is a $\mathbb{Z}$-basis for $R_2$.
\end{lem}

\begin{proof}
    $\B_2$ generates all $R_2$, since by relation $(1)$ $ \gamma_A^i\gamma_A^i=0$, by relation $(3)$ $\gamma_A^i\gamma_B^j=0$ if $A,B$ are both the respective dominant components, and if $A=B=\C$ then by relation $(4)$ $\gamma_\C^i\gamma_C^j$ can be obtained as a $\mathbb{Z}$-linear combination of $\gamma_\C^i\gamma_{D^i}^j$ and $\gamma_{D^j}^i\gamma_C^j$. $\B_2$ is clearly linearly independent since there are no other relations left to make a linear combination zero if the coefficients are not all zero.
\end{proof}

\begin{nota}
    If we have two compatible based partitions $\tau_i,\tau_j$ and a petal $P\in\tau_i$ does not contain the dominant component $D^j$, then $P$ can always be split if it is formed by more than one connected component, and the new based partition $\tau_i'$ will still be compatible with $\tau_j$. Recall as well that if $\G-\st(v_i)$ and $\G-\st(v_j)$ have different minimal elements, then one of the minimal elements must lie in the respective dominant component.
\end{nota}

We will state two technical lemmas.

\begin{lem}\label{lem:minimal-component}
    Let $v_i\notin\st(v_j)$ and let $\min_i$ denote the minimal element in $\G-\st(v_i)$. Let $\mathcal{S}$ denote a subordinate component.  If $\min_i\neq\min_j$ then either $\min_i\in D^j$, or $\min_j\in D^i$ or both. If  $\min_i=\min_j$ then either $\min_i\in D^j$ and $\min_j\in \mathcal{S}\subset D^j$,  or $\min_j\in D^i$ and $\min_i\in \mathcal{S}\subset D^i$,  or $\min_i$ and $\min_j$ are in the same shared component $\C$.
\end{lem}
\begin{proof}
    Assume that $\min_i\neq\min_j$. Then, one of them, say $\min_j$, is smaller. Therefore $\min_j\notin(\G-\st(v_i))$, so $\min_j\in\st(v_i)$ and $\min_j\in D^i$. The second claim follows  from the next fact: if $A,B$ are connected components in $\G-\st(v_i)$ and    $\G-\st(v_j)$ respectively, and $A\cap B\neq\emptyset$, then either $A=B$ and they are both the same shared component or $A\subset B$ and $A=\mathcal{S}$ and $B=D^i$.
    
\end{proof}

\begin{lem}\label{lem:rank2Partitions}
    Let $\tau$ be a rank 2 essential vertex type in $\Wh_\G$ with two non trivial based partitions $\tau_i=\{\{v_i\}, P_1, P_2\}$ and $\tau_j=\{\{v_j\}, Q_1, Q_2\}$. Assume that $\min_i\in P_1$ and $\min_j\in Q_1$. Then either $P_2$ and $Q_2$ both consist of a unique connected component, or $P_2=\C$ is a shared component and $Q_2=D^i\cup\C$ (or the other way around).
\end{lem}
\begin{proof}
    Since $\tau$ is essential, the only petals admitting splits are those containing the minimal elements, so neither $P_2$ nor $Q_2$ admit a split. If $v_i\in\lk(v_j)$ then based partitions are always compatible, hence, since $P_2$ and $Q_2$ do not admit any split they must consist of a unique connected component. If $v_i\notin\lk(v_j)$, we claim that at least one of the $P_1,Q_1$ contains the respective dominant component: if the minimal elements are different we apply the previous lemma; if the minimal element is the same in both $\tau_i$ and $\tau_j$, we  have  $P_1\cap Q_1\neq\emptyset$, so at least one of $P_1, Q_1$ contains the dominant component, because in other case we would have a crossing and the partitions would not be compatible. Therefore, we can assume that $D^j\subseteq P_1$. Since $P_2$ does not contain the dominant component, and it does not admit a split, then it must consist of a unique connected component. If $Q_2$ is formed by a unique connected component we are done. If not, $Q_2=\{D^i\cup \mathcal{C}_1\cup\cdots \cup \mathcal{C}_k\}$, where each $\mathcal{C}_i$ is shared, because subordinate components do not produce crossings and can always be split. But these $\mathcal{C}_1,...,\mathcal{C}_k$ must be in $P_2$, to impede the split of $Q_2$. As a consequence,  since $P_2$ consists of a unique connected component, we have $P_2=\C$ and $Q_2=D^i\cup \C$.
\end{proof}

\begin{prop}
    There exists an surjective map   $\varphi: \B_1\rightarrow\B_2$. \label{prop:baseIsomor}
\end{prop}
\begin{proof}
We will start by  constructing $\varphi$ explicitly for each of the types of elements of $\B_1$ considering several possible cases. During the construction, we will denote by boldface  $\boldsymbol{A}$ the petal $A$ containing the minimal element $\min_i$ in the based partition $\tau_i$. 

    \noindent{\textbf{Type 1:}} Rank 2 vertex types $\tau$.

     {\textbf{1.1:}} $\tau$ has a single non-trivial based partition $\tau_i=\{\{v_i\},P_1,P_2,P_3\}$:  We may assume that  $\min_i\in P_1$. Then, since $\tau$ is essential, $P_2$ and $P_3$ cannot admit a split, so they must consist of a unique connected component, because otherwise they could be split, since all of the other based partitions are trivial. We set 
        $$
        \varphi(\tau)= \gamma^i_{P_2}\gamma^i_{P_3}\in\B_2.
        $$

     {\textbf{1.2:}} $\tau$ has two non-trivial based partitions $\tau_i=\{\{v_i\},P_1,P_2\}$ and $\tau_j=\{\{v_j\},Q_1,Q_2\}$:
     We  assume that that $\min_i\in P_1$ and $\min_j\in Q_1$. By Lemma \ref{lem:rank2Partitions}, we have two possible cases:
        
         {\textbf{1.2.1:}}  $v_i\notin \lk(v_j)$, $P_2=\C$ a shared component and $Q_2=D^i\cup\C$. We set
         $$\varphi(\tau)= \gamma^i_{P_2}\gamma^j_{D^i}\in\B_2$$

         {\textbf{1.2.2:}} In other case, since $P_2$ and $Q_2$ are formed by a unique connected component, we set:
        $$\varphi(\tau)= \gamma^i_{P_2}\gamma^j_{Q_2}\in\B_2.$$
        Note that $P_2$ and $Q_2$ must be different, because if $P_2=\C=Q_2$ there would be a crossing.

\noindent 

\noindent{\textbf{Type 2:}} Pairs $(\tau,v_j)$ for $\tau$ a rank 1 essential vertex and $v_j\in\G$. In $\tau$ every based partition is trivial except for  $\tau_i=\{\{v_i\},P_1,P_2\}$. We  assume that   $\min_i\in P_1$, so $P_2$ must consist of a unique connected component; in other case, it would admit a split. Again, we consider several subcases. 

{\textbf{2.1:}}  Assume $v_i\notin \st (v_j)$,  $\min_i\neq\min_j$ one of the $\min_i,\min_j$ is in a shared component and the other in a dominant component.

{\textbf{2.1.1:}}  If $\min_j\in\C$, a shared component, then, by Lemma \ref{lem:minimal-component},  $\min_i\in D^j$. For the pair $(\tau,v_j)$ with $\tau_i=\{\{v_i\},P_1,\C\}$, we set:
$$
\varphi(\tau,v_j)=\gamma^i_{\boldsymbol{D}^j}\gamma^j_{\boldsymbol{\C}}
$$

{\textbf{2.1.2:}} If $\min_i\in\C$, a shared component, then, by Lemma \ref{lem:minimal-component},  $\min_j\in D^i$. For the pair $(\tau,v_j)$ with $\tau_i=\{\{v_i\},P_1,D^j\}$, we set:
$$
\varphi(\tau,v_j)=\gamma^i_{D^j}\gamma^j_\C
$$

{\textbf{2.2:}} Assume $v_i\notin \st (v_j)$, $\min_j\in D^i$,   $\min_i\in\mathcal{S}$, a subordinate component, and the only non-trivial based partition in $\tau$ is $\tau_i=\{\{v_i\},P_1,D^j\}$. We set:
 $$
    \varphi(\tau,v_j)= \gamma^i_{\boldsymbol{\s}}\gamma^j_{\boldsymbol{D}^i}
    $$

    {\textbf{2.3:}} In any other case, for the pair $(\tau,v_j)$ with $\tau_i=\{\{v_i\},P_1,P_2\}$ such that  $\min_i\in P_1$ and  $\min_j\in B$, we set:
    $$
    \varphi(\tau,v_j)= \gamma^i_{P_2}\gamma^j_{\boldsymbol{B}}
    $$

\noindent{\textbf{Type 3:}} Edges $(v_i,v_j)\in E(\G)$. Consider the connected components $A$ of $\G-\st(v_i)$ and $B$ of $\G-\st(v_j)$  that contain the minimal elements and set
 $$
    \varphi(i,j)= \gamma^i_{\boldsymbol{A}}\gamma^j_{\boldsymbol{B}}
    $$

During the construction we have already checked that the elements chosen in the image of $\varphi$ lie in $\B_2$; equivalently,  that they are of the form $\gamma^i_{A}\gamma^j_{B}$ where  $A$ and $B$ consist of a unique  connected component in $\G-\st(v_i)$ and $\G-\st(v_j)$ respectively, $A\neq B$, and $A,B$ are not the respective dominant components at the same time. 

It remains to prove that the map is surjective. Let $\gamma^i_{A}\gamma^j_{B}\in\B_2$.


\noindent{\textbf{Case I:}}  Assume $\min_i\notin A$ and $\min_j\in B$. Let $\tau$ be the rank 1 vertex type with only non trivial based partition  $\tau_i=\{\{v_i\}, A^C, A\}$, here $A^C$ denotes $(\G-\st(v_i))-A$. We have that $\tau$ is essential, since $A$ consists of a unique connected component. Hence, $(\tau,v_j)$ is of type 2.3 and $\varphi(\tau,v_j)=\gamma^i_{A}\gamma^j_{B}$.

\noindent{\textbf{Case II:}} Assume $\min_i\in A$ and $\min_j\in B$ (so $v_i\neq v_j$).
\begin{itemize}
    \item If $v_i\in\lk(v_j)$, then $(v_i,v_j)$ is of type 3 and $\varphi(v_i,v_j)=\gamma^i_{A}\gamma^j_{B}$.
    \item If $v_i\notin\lk(v_j)$, then $A$ and $B$ cannot be shared and equal, and they cannot be both dominant. By Lemma \ref{lem:minimal-component}, one must be dominant and the other is either shared or subordinate.

    (1) Assume $A=D^j$ dominant and $B$ shared. Let $\tau$ be the rank 1 vertex type with only non trivial based partition  $\tau_i=\{\{v_i\}, B^C, B\}$. We have that $\tau$ is essential, since $B$ consists of a unique connected component (note that $\min_i\notin B$). Hence, $(\tau,v_j)$ is of type 2.1.1 and $\varphi(\tau,v_j)=\gamma^i_{A}\gamma^j_{B}$.

    (2) Assume $B=D^i$ dominant and $A$ subordinate. Let $\tau$ be the rank 1 vertex type with only non trivial based partition  $\tau_i=\{\{v_i\}, (D^j)^C, D^j\}$. Again, $\tau$ is essential. Hence, $(\tau,v_j)$ is of type 2.2 and $\varphi(\tau,v_j)=\gamma^i_{A}\gamma^j_{B}$.
\end{itemize}

\noindent{\textbf{Case III:}} Assume $\min_i\notin A$ and $\min_j\notin B$.
\begin{itemize}
    \item If $v_i=v_j$, let  $\tau$ be the rank 2 essential vertex type with only non trivial based partition  $\tau_i=\{\{v_i\}, (A\cup B)^C, A, B\}$. Then, $\tau$ is of type 1.1 and $\varphi(\tau)=\gamma^i_{A}\gamma^i_{B}$.
    \item Assume that  $\tau_i=\{\{v_i\}, A^C, A\}$ and $\tau_j=\{\{v_j\}, B^C, B\}$ are compatible. Then the rank 2 vertex type $\tau$ with only non trivial based partitions  $\tau_i$ and $\tau_j$ is essential, so it is of type 1.2.2 and $\varphi(\tau)=\gamma^i_{A}\gamma^j_{B}$.
    \item In other case: the fact that $\tau_i$ and $\tau_j$ are not compatible implies that one of $A,B$  must be  dominant and the other shared. We assume that $A=D^j$, so $B$ is shared.

    (1) If  $\min_i\in B$, let $\tau$ be the rank 1 essential vertex type with only non trivial based partition  $\tau_i=\{\{v_i\}, A^C, A\}$. Then, $(\tau,v_j)$ is of type 2.1.2 and $\varphi(\tau,v_j)=\gamma^i_{A}\gamma^j_{B}$.

    (2) If  $\min_i\notin B$, let $\tau$ be the rank 1 vertex type with only non trivial based partitions $\tau_i=\{\{v_i\}, (A\cup B)^C, A\cup B\}$ and $\tau_j=\{\{v_j\}, B^C,  B\}$. The petal $A$  does not admit a split because a crossing would appear, so $\tau$ is essential. Hence, $\tau$ is of type 1.2.1 and $\varphi(\tau)=\gamma^j_{B}\gamma^i_{A}$.

\end{itemize}

\end{proof}

So we have
\begin{teo} The map $\phi$ of Corollary  \ref{cor:H2_1} is an isomorphism in degrees 1 and 2, so, in particular
    $$
H^2(\PAut(A_\G))\cong R_2.
    $$
\end{teo}
\begin{proof}
    The fact that $\phi$ is surjective implies that $\textrm{dim}(R_2)\geq \textrm{dim}(H^2(\PAut(A_\G))$ and Proposition \ref{prop:baseIsomor} implies that both dimensions are equal.
\end{proof}

\begin{conj}[Generalized Brownstein-Lee Conjecture]
The map $\phi$ of Corollary \ref{cor:H2_1} is an isomorphism.
\end{conj}

\section{The Torelli subgroup}

Let $G$ be a group, the Torelli subgroup  $\IA(G)$  is the kernel of the  homomorphism:
$$
\phi:\mathrm{Aut}(G)\rightarrow \mathrm{Aut}(G^{ab})\cong \mathrm{Aut}(H_1(G)).
$$
In general, not much is known about these groups. In the case of a RAAG $\A$, Day found in \cite{Day}  a finite generating system for $\IA(\A)$, and Wade proved in his PhD thesis \cite{Wade} that $H_1(\IA(\A))$ is the free abelian group on that set of generators. Following the idea of Corollary 6.9 in \cite{JMcCM} (pointed by Fred Cohen), we can obtain some information about $H^*(\IA(\A))$ from Theorem B.

\begin{cor}
The injection $\PAut(\A)\hookrightarrow\IA(\A)$ induces a split epimorphism
$$
    H^*(\IA(\A)) \twoheadrightarrow H^*(\PAut(A_\G)).
    $$
Furthermore, the suspension of $B\PAut(\A)$ is homotopic to a bouquet of spheres, and it is a retract of the suspension of $B\IA(\A)$.

    \label{cor:IA}
\end{cor}

\begin{proof}
    Once it has been proven that $H^*(\PAut(\A))$ is generated by the one-dimensional classes (Corollary \ref{cor:onedimPAut}), and that $H_1(\IA(\A))$ is free abelian (Theorem 4.23 in \cite{Wade}), the result follows by the same arguments as in  Corollary 6.9 in \cite{JMcCM}.
\end{proof}

\appendix

\section{Algorithm to compute the $K_q$:} \label{appendix}

The main theorems of this paper (Theorem A and B) give us a formula to compute the cohomology groups $H^q(\POut(\A))$ and $H^q(\POut(\A))$. In the case when $\A$ is the free group $F_n$, it is possible to give a closed formula in terms of $n$ only (Lemma 5.2 in \cite{JMcCM}). However, that cannot be done for arbitrary RAAGs; we can, nonetheless, implement an algorithm that can be used for explicit computations. 

The idea is simple: for each vertex $v_i\in\G$, we take all the connected components in $\G-\st(v_i)$ and combine them in all possible ways to obtain $\mathcal{P}_i$, the set of the based partitions associated to each $v_i$. Then, we consider the cartesian product of the $\mathcal{P}_i$'s. We have to remove from the cartesian product the elements that are not vertex types, so for each SIL-pair $(v,w)$ we check which based partitions cross and we remove  all the elements which have those two based partitions at the same time.

Once we have all the vertex types, we select those of rank $q$  and check if there is some based partition which does not contain the minimal element and admits a split; if there is none, then the vertex type is essential.


A code implemented for Python which computes the Whitehead poset, the number of essential vertices for each rank, the number of simplices in $|\Wh_\G|$ and the $E^1$ page for the $\POut(\A)$ action on $\MM_\G$ can be found at \url{https://github.com/peioardaiz-gale/Codes-for-the-generalized-Whitehead-poset/blob/main/whitehead-poset_essential-vertices_E1-page.py}.

\begin{algorithm}[H]
    
\caption{Algorithm to compute $K_q$, the number of essential vertex types of rank $q$}
\begin{algorithmic}[1]

    \For {$1\leq q\leq n-2$}
    \State $K_q\gets 0$
    \EndFor
    \For {$1\leq i\leq n$}
    \State Compute a list $\mathcal{P}_i$ with all the valid based partitions  of $\G-\st(v_i)$
    \EndFor
     \For {$\tau\in\mathcal{T}$}
        \If {the based partitions in $\tau\in \mathcal{T}$ are pairwise compatible}
            \State $\tau \in \Wh_\G$
        \EndIf
    \EndFor
    \For {$1\leq q\leq $ MAX$(rk(\tau))$}
    \State $K_q\gets 0$
        \For {$\tau\in\Wh_\G$}
        \If {$rk(\tau)=q$ }
            \If {$\tau$ is essential}
            \State $K_q++$
            \EndIf
        \EndIf
        \EndFor
    \EndFor

\end{algorithmic}

\end{algorithm}

\end{document}